\documentclass[12pt,reqno]{amsart}
\usepackage{amsthm, amscd, amsfonts, amssymb, graphicx, color}
\usepackage[bookmarksnumbered, colorlinks, plainpages,linkcolor=green,anchorcolor=blue,citecolor=blue,urlcolor=blue]{hyperref}
\usepackage{enumerate}
\usepackage{exscale}
\usepackage[all]{xy}
\usepackage{relsize}
\usepackage{pdfsync}
\numberwithin{equation}{section}

\makeatletter
\@namedef{subjclassname@2020}{%
	\textup{2020} Mathematics Subject Classification}
\makeatother

\usepackage[includemp,body={398pt,550pt},footskip=30pt,%
marginparwidth=60pt,marginparsep=10pt]{geometry}

\newtheorem{theorem}{Theorem}[section]
\newtheorem{lemma}[theorem]{Lemma}
\newtheorem{proposition}[theorem]{Proposition}
\newtheorem{corollary}[theorem]{Corollary}
\theoremstyle{definition}

\numberwithin{equation}{section}

\newcommand{\B}{\mathbb{B}}
\newcommand{\h}{\mathcal{H}}

\newcommand{\C}{\mathbb{C}}
\newcommand{\s}{\mathbb{S}}
\newcommand{\HT}{\mathcal{HT}}

\begin{document}
\title[Tent Carleson measures and superposition operators]{Tent Carleson measures and superposition operators on Hardy type tent spaces}
\date{June 27, 2024.
}
\author[J. Chen]{Jiale Chen}
\address{School of Mathematics and Statistics, Shaanxi Normal University, Xi'an 710119, China.}
\email{jialechen@snnu.edu.cn}

\thanks{This work was supported by the Fundamental Research Funds for the Central Universities (No. GK202207018) of China.}



\begin{abstract}
  \noindent In this paper, we completely characterize the positive Borel measures $\mu$ on the unit ball $\mathbb{B}_n$ of $\mathbb{C}^n$ such that the Carleson embedding from holomorphic Hardy type tent spaces $\mathcal{HT}^p_{q,\alpha}$ into the tent spaces $T^t_s(\mu)$ is bounded for all $0<p,q,s,t<\infty$ and $\alpha>-n-1$. As an application, we determine the indices $0<p,q,s,t<\infty$ and $\alpha,\beta>-n-1$ such that the inclusion $\mathcal{HT}^p_{q,\alpha}\subset\mathcal{HT}^t_{s,\beta}$ holds, which allows us to completely characterize the nonlinear superposition operators between Hardy type tent spaces.
\end{abstract}
\maketitle


\section{Introduction}
\allowdisplaybreaks[4]
Let $\mathbb{B}_n$ be the open unit ball of $\mathbb{C}^n$ and $\s_n=\partial\B_n$ be the unit sphere. For $\xi\in\s_n$ and $\gamma>1$, let $\Gamma_{\gamma}(\xi)$ denote the Kor\'{a}nyi region defined by
$$\Gamma_{\gamma}(\xi):=\left\{z\in\B_n:|1-\langle z,\xi\rangle|<\frac{\gamma}{2}\left(1-|z|^2\right)\right\}.$$
In this paper, we agree that $\Gamma(\xi):=\Gamma_2(\xi)$. Given $0<p,q<\infty$ and a positive Borel measure $\mu$ on $\B_n$, the tent space $T^p_q(\mu)$ consists of $\mu$-measurable functions $f:\B_n\to\C$ such that
$$\|f\|_{T^p_q(\mu)}:=\left(\int_{\s_n}\left(\int_{\Gamma(\xi)}|f(z)|^q\frac{d\mu(z)}{(1-|z|^2)^n}\right)^{p/q}d\sigma(\xi)\right)^{1/p}
<\infty,$$
where $d\sigma$ is the normalized area measure on $\s_n$. If $d\mu(z)=dv_{n+\alpha}(z):=(1-|z|^2)^{n+\alpha}dv(z)$ for some $\alpha\in\mathbb{R}$, then we write $T^p_{q,\alpha}:=T_q^p(v_{\alpha+n})$. Here $dv$ denotes the normalized Lebesgue measure on $\B_n$. In the above definition, the Kor\'{a}nyi region $\Gamma(\xi)$ can be replaced by $\Gamma_{\gamma}(\xi)$ for any aperture $\gamma>1$, and it is well-known that any two apertures generate the same function space with equivalent (quasi-)norms.

The concept of tent spaces was first introduced by Coifman, Meyer and Stein \cite{CMS} in order to study several problems in harmonic analysis. Since then, these spaces have become powerful tools in harmonic analysis and operator theory; see \cite{CV,Lue91,MPPW,P1,P2018,PR}. In 2018, Per\"{a}l\"{a} \cite{Per} initiated the study of the holomorphic versions of tent spaces
on the unit ball $\B_n$. Given $0<p,q<\infty$ and $\alpha>-n-1$, the holomorphic Hardy type tent space $\HT^p_{q,\alpha}$ is defined by
$$\HT^p_{q,\alpha}:=T^p_{q,\alpha}\cap\h(\B_n)$$
with inherited (quasi)-norm, where $\h(\B_n)$ is the set of holomorphic functions on $\B_n$. It is quite interesting that these holomorphic tent spaces include several classical spaces of holomorphic functions on $\B_n$. For $0<p<\infty$ and $\alpha>-1$, the tent space $\HT^p_{p,\alpha-n}$ is exactly the Bergman space $A^p_{\alpha}$, consisting of holomorphic functions $f$ on $\B_n$ such that
$$\|f\|^p_{A^p_{\alpha}}:=\int_{\mathbb{B}_n}|f(z)|^pdv_{\alpha}(z)<\infty;$$
and a function $g\in\h(\B_n)$ belongs to the Hardy space $H^p$ if and only if $Rg\in\HT^{p}_{2,1-n}$, where $Rg$ is the radial derivative of $g$ (see for instance \cite{P1}).

Among other results, Per\"{a}l\"{a} \cite{Per} described the dual of $\HT^p_{q,\alpha}$ for all $0<p,q<\infty$ and $\alpha>-n-1$. After that, the function theory and operator theory on these holomorphic tent spaces developed quickly; see \cite{Ch23,Pa20,WZ21,WZ22}. In this paper, we are going to establish a tent Carleson measure theorem for Hardy type tent spaces $\HT^p_{q,\alpha}$, and then apply it to characterize the nonlinear superposition operators acting between $\HT^p_{q,\alpha}$ and $\HT^t_{s,\beta}$.

The concept of Carleson measures was introduced by Carleson \cite{Car0,Car1} when studying interpolating sequences for bounded analytic functions on the unit disk, in route of solving the famous corona problem. Carleson's original result characterizes the positive Borel measures $\mu$ on the unit disk $\mathbb{D}$ such that the embedding $I_d:H^p\to L^p(\mathbb{D},\mu)$ is bounded for $0<p<\infty$, which was  extended to the setting of the unit ball $\B_n$ by Duren, H\"{o}rmander and Luecking \cite{Du69,H,Lue91}, by obtaining a complete description of the boundedness of the embedding $I_d:H^p\to L^q(\B_n,\mu)$ for all $0<p,q<\infty$. After that, Carleson type measures for various function spaces have been investigated and have become important tools for the study of function and operator theory; see for instance \cite{KS88,PP10,PZ15,PR14}. Recently, Wang and Zhou \cite{WZ21,WZ22} characterized the boundedness of the embedding $I_d:\HT^p_{q,\alpha}\to L^s(\B_n,\mu)$ and applied it to investigate the boundedness of Volterra type integration operators and Toeplitz type operators acting on Hardy type tent spaces.

Apart from embedding Hardy type tent spaces $\HT^p_{q,\alpha}$ into Lebesgue spaces $L^s(\B_n,\mu)$, it is more natural to consider the embedding $I_d:\HT^p_{q,\alpha}\to T^t_s(\mu)$, that we are going to pursue in this paper. It turns out that this is equivalent to consider area operators acting from $\HT^p_{q,\alpha}$ into Lebesgue spaces $L^t(\s_n)$. Given $s>0$ and a positive Borel measure on $\B_n$, the area operator $A_{\mu,s}$ acting on $\h(\B_n)$ is defined by
$$A_{\mu,s}f(\xi):=\left(\int_{\Gamma(\xi)}|f(z)|^s\frac{d\mu(z)}{(1-|z|^2)^n}\right)^{1/s},\quad \xi\in\s_n.$$
It is obvious that $I_d:\HT^p_{q,\alpha}\to T^t_s(\mu)$ is bounded if and only if $A_{\mu,s}:\HT_{q,\alpha}^p\to L^t(\s_n)$ is bounded. The operators $A_{\mu,s}$ were introduced by Cohn \cite{Co} in the setting of the unit disk, and related results can be found in \cite{GLW,LLZ,Wu06,Wu11}. Lately, Lv, Pau and Wang \cite{LvP2,LvPW} completely characterized the boundedness of area operators acting on Hardy and Bergman spaces over the unit ball $\B_n$. Our first result generalizes these results and characterizes the boundedness of $I_d:\HT^p_{q,\alpha}\to T^t_s(\mu)$, as well as $A_{\mu,s}:\HT_{q,\alpha}^p\to L^t(\s_n)$, for all possible choices of $0<p,q,s,t<\infty$ and $\alpha>-n-1$.

To state our first result, recall that the Bergman metric $\beta(\cdot,\cdot)$ on $\B_n$ is defined by
$$\beta(z,w):=\frac{1}{2}\log\frac{1+|\varphi_z(w)|}{1-|\varphi_z(w)|},\quad z,w\in\B_n,$$
where $\varphi_z$ is the involutive automorphism of $\B_n$ that interchanges $z$ and $0$. Given $\alpha>-n-1$, $r\in(0,1)$ and a positive Borel measure $\mu$ on $\B_n$, we define
$$\widehat{\mu}_r(z):=\frac{\mu(D(z,r))}{(1-|z|^2)^{2n+1+\alpha}},\quad z\in\B_n,$$
where $D(z,r):=\{w\in\B_n:\beta(z,w)<r\}$ is the Bergman metric ball centered at $z$ with the radius $r$. Our first result reads as follows.

\begin{theorem}\label{main}
Let $\alpha>-n-1$, $0<p,q,s,t<\infty$, and let $\mu$ be a positive Borel measure on $\B_n$. Then the following statements are equivalent.
\begin{enumerate}
	\item [(a)] $I_d:\HT^p_{q,\alpha}\to T^t_s(\mu)$ is bounded.
	\item [(b)] $A_{\mu,s}:\HT_{q,\alpha}^p\to L^t(\s_n)$ is bounded.
	\item [(c)] For any $r\in(0,1)$, one of the following conditions holds:
	  \begin{enumerate}
	  	\item [(1)] If $p<t$, or $p=t$ and $q\leq s$, then
	  	$$\sup_{z\in\B_n}\widehat{\mu}_r(z)(1-|z|^2)^{\frac{(q-s)(n+1+\alpha)}{q}+ns\left(\frac{1}{t}-\frac{1}{p}\right)}<\infty;$$
	  	\item [(2)] If $p=t$ and $q>s$, then
	  	$$\widehat{\mu}_r(z)^{\frac{q}{q-s}}dv_{\alpha+n}(z)$$
	  	is a Carleson measure;
	  	\item [(3)] If $p>t$ and $q>s$, then
	  	$$\xi\mapsto\left(\int_{\Gamma(\xi)}\widehat{\mu}_r(z)^{\frac{q}{q-s}}dv_{\alpha}(z)\right)^{\frac{q-s}{q}}$$
	  	belongs to $L^{\frac{pt}{s(p-t)}}(\s_n)$;
	  	\item [(4)] If $p>t$ and $q\leq s$, then
	  	$$\xi\mapsto\sup_{z\in\Gamma(\xi)}\widehat{\mu}_r(z)(1-|z|^2)^{\frac{(q-s)(n+1+\alpha)}{q}}$$
	  	belongs to $L^{\frac{pt}{s(p-t)}}(\s_n)$.
	  \end{enumerate}
\end{enumerate}
\end{theorem}

Our approach to Theorem \ref{main} is mainly based on Carleson measures and Khinchine--Kahane's inequalities, dual and factorization tricks for sequence tent spaces, as well as some inclusions between Hardy type tent spaces and Bergman spaces. As an immediate application, we can characterize the inclusion between Hardy type tent spaces as follows.

\begin{corollary}\label{incl}
Let $\alpha,\beta>-n-1$ and $0<p,q,s,t<\infty$. Then $\HT^p_{q,\alpha}\subset\HT^t_{s,\beta}$ if and only if one of the following conditions holds:
\begin{enumerate}
	\item [(i)] $p\geq t$, $q>s$ and $\frac{n+1+\alpha}{q}<\frac{n+1+\beta}{s}$;
	\item [(ii)] $p\geq t$, $q\leq s$ and $\frac{n+1+\alpha}{q}\leq\frac{n+1+\beta}{s}$;
	\item [(iii)] $p<t$ and $\frac{n+1+\alpha}{q}+\frac{n}{p}\leq\frac{n+1+\beta}{s}+\frac{n}{t}$.
\end{enumerate}
\end{corollary}

With Corollary \ref{incl} in hand, it is possible to determine the nonlinear superposition operators between Hardy type tent spaces. Given a function $\varphi$ on the complex plane $\C$, the nonlinear superposition operator $S_{\varphi}$ is defined for complex-valued functions $f$ by
$$S_{\varphi}f=\varphi\circ f.$$
Given two function spaces $X$ and $Y$, a natural question is to find the functions $\varphi$ such that the operators $S_{\varphi}$ map $X$ into $Y$. The theory of superposition operators has a long history in the context of real-valued functions; see \cite{AZ}. Since the pioneering works of C\'{a}mera and Gim\'{e}nez \cite{Ca95,CaG94}, the superposition operators on spaces of analytic functions have drawn lots of attention; see \cite{AMV,DoGi,Me22,SW} and the references therein. However, as far as we know, it seems that the superposition operators have not been studied in the setting of several complex variables. We here invoke Corollary \ref{incl} to characterize the superposition operators between Hardy type tent spaces as below.

\begin{theorem}\label{super}
Let $0<p,q,s,t<\infty$ and $\alpha,\beta>-n-1$, and let $\varphi$ be a function on $\C$. Then the superposition operator $S_{\varphi}$ maps $\HT^{p}_{q,\alpha}$ into $\HT^{t}_{s,\beta}$ if and only if $\varphi$ is a polynomial of degree $N$, where $N$ satisfies
\begin{enumerate}
	\item [(i)] $N\leq\frac{pq(t(n+1+\beta)+ns)}{st(p(n+1+\alpha)+nq)}$ in the case of $p(n+1+\alpha)/q<t(n+1+\beta)/s$;
	\item [(ii)] $N\leq\frac{q(n+1+\beta)}{s(n+1+\alpha)}$ in the case of $p(n+1+\alpha)/q\geq t(n+1+\beta)/s$ and $\alpha\leq\beta$;
	\item [(iii)] $N<\frac{q(n+1+\beta)}{s(n+1+\alpha)}$ in the case of $p(n+1+\alpha)/q\geq t(n+1+\beta)/s$ and $\alpha>\beta$.
\end{enumerate}
Moreover, if $S_{\varphi}$ maps $\HT^{p}_{q,\alpha}$ into $\HT^{t}_{s,\beta}$, then it is bounded and continuous.
\end{theorem}

Since the Bergman space $A^p_{\alpha}$ is exactly $\HT^p_{p,\alpha-n}$ for $\alpha>-1$, we obtain the following description of superposition operators between Bergman spaces over the unit ball $\B_n$, which generalizes \cite[Theorem 1]{CaG94} to the setting of higher dimensions.

\begin{corollary}
Let $\alpha,\beta>-1$ and $0<p,q<\infty$, and let $\varphi$ be a function on $\C$. Then the superposition operator $S_{\varphi}$ maps $A^p_{\alpha}$ into $A^q_{\beta}$ if and only if $\varphi$ is a polynomial of degree $N$, where $N$ satisfies
\begin{enumerate}
	\item [(i)] $N\leq\frac{p(n+1+\beta)}{q(n+1+\alpha)}$ in the case $\alpha\leq\beta$;
	\item [(ii)] $N<\frac{p(\beta+1)}{q(\alpha+1)}$ in the case $\alpha>\beta$.
\end{enumerate}
Moreover, if $S_{\varphi}$ maps $A^p_{\alpha}$ into $A^q_{\beta}$, then it is bounded and continuous.
\end{corollary}

This paper is organized as follows. In Section \ref{Pre}, we display some background and preliminary results. Section \ref{proof} is devoted to the proof of Theorem \ref{main}. In Section \ref{compact}, we characterize the compactness of the embedding $I_d:\HT^p_{q,\alpha}\to T^t_s(\mu)$. Finally, in Section \ref{su}, we prove Theorem \ref{super}.

Throughout the paper, the notation $A\lesssim B$ (or $B\gtrsim A$) means that  $A\leq CB$ for some inessential constant $C>0$. If $A\lesssim B\lesssim A$, then we write $A\asymp B$. For $1\leq p<\infty$, we use $p'$ to denote the conjugate exponent of $p$, i.e. $\frac{1}{p}+\frac{1}{p'}=1$. For a subset $E\subset\B_n$, $\chi_E$ denotes its characteristic function. For any $\varrho\in(0,1)$, we write $\varrho\B_n=\{z:|z|<\varrho\}$ and $(\varrho\B_n)^c=\B_n\setminus\varrho\B_n$.

\section{Preliminaries}\label{Pre}

In this section, we introduce some preliminary results that will be used throughout the paper.

\subsection{Carleson measures}
For $\xi\in\mathbb{S}_n$ and $\delta>0$, the non-isotropic metric ball $B_{\delta}(\xi)$ is defined by
$$B_{\delta}(\xi)=\{z\in\mathbb{B}_n:|1-\langle z,\xi\rangle|<\delta\}.$$
A positive Borel measure $\mu$ on $\mathbb{B}_n$ is said to be a Carleson measure if there is a constant $C>0$ such that
$$\mu(B_{\delta}(\xi))\leq C\delta^n  $$
for all $\xi\in\mathbb{S}_n$ and $\delta>0$. We denote by $\|\mu\|_{CM}$ the infimum of all possible $C$ above. Obviously every Carleson measure is finite. The famous Carleson measure embedding theorem \cite{Car0, Car1}, extended to several complex variables by H\"{o}rmander \cite{H}, asserts that for $0<p<\infty$, the embedding $I_d:H^p\to L^p(\B_n,d\mu)$ is bounded if and only if $\mu$ is a Carleson measure.  It is also known (see \cite[Theorem 45]{ZZ}) that $\mu$ is a Carleson measure if and only if for each (or some) $\theta>0$,
\begin{equation}\label{CM}
\sup_{a\in\mathbb{B}_n}\int_{\mathbb{B}_n}\frac{(1-|a|^2)^{\theta}}{|1-\langle z,a\rangle|^{n+\theta}}d\mu(z)<\infty.
\end{equation}
Moreover, with constant depending on $t$, the supremum of the above integral is comparable to $\|\mu\|_{CM}$.

A positive Borel measure $\mu$ on $\mathbb{B}_n$ is called a vanishing Carleson measure if
$$\lim_{\delta\rightarrow0}\frac{\mu(B_{\delta}(\xi))}{\delta^n}=0$$
uniformly for $\xi\in\mathbb{S}_n$. Equivalently, one may require that for each (or some) $\theta>0$,
\begin{equation}\label{VCM}
\lim_{|a|\rightarrow1^-}\int_{\mathbb{B}_n}\frac{(1-|a|^2)^{\theta}}{|1-\langle z,a\rangle|^{n+\theta}}d\mu(z)=0.
\end{equation}

We will need the following elementary characterization for vanishing Carleson measures.

\begin{lemma}\label{rVCM}
Let $\mu$ be a Carleson measure on $\B_n$. Then $\mu$ is a vanishing Carleson measure if and only if
$$\lim_{\varrho\to1^-}\|\chi_{(\varrho\mathbb{B}_n)^c}\mu\|_{CM}=0.$$
\end{lemma}
\begin{proof}
The necessity can be found in \cite[Lemma 3.1]{CPW}. Suppose now that
$$\lim_{\varrho\to1^-}\|\chi_{(\varrho\mathbb{B}_n)^c}\mu\|_{CM}=0.$$
Then for any $\epsilon>0$, there exists $\varrho_0\in(0,1)$ such that $\|\chi_{(\varrho_0\mathbb{B}_n)^c}\mu\|_{CM}<\epsilon$. Consequently, for $\theta>0$ and any $a\in\B_n$,
\begin{align*}
\int_{\B_n}\frac{(1-|a|^2)^{\theta}}{|1-\langle z,a\rangle|^{n+\theta}}d\mu(z)
&=\left(\int_{\varrho_0\B_n}+\int_{(\varrho_0\B_n)^c}\right)\frac{(1-|a|^2)^{\theta}}{|1-\langle z,a\rangle|^{n+\theta}}d\mu(z)\\
&\lesssim\frac{(1-|a|^2)^{\theta}}{(1-\varrho_0)^{n+\theta}}\mu(\B_n)+\|\chi_{(\varrho_0\mathbb{B}_n)^c}\mu\|_{CM},
\end{align*}
which implies that
$$\limsup_{|a|\to1^-}\int_{\B_n}\frac{(1-|a|^2)^{\theta}}{|1-\langle z,a\rangle|^{n+\theta}}d\mu(z)
\lesssim\|\chi_{(\varrho_0\mathbb{B}_n)^c}\mu\|_{CM}<\epsilon.$$
The arbitrariness of $\epsilon$ together with the characterization \eqref{VCM} gives that $\mu$ is a vanishing Carleson measure.
\end{proof}

\subsection{Area methods and some inclusions}
Recall that $D(z,\delta)$ is the Bergman metric ball centered at $z\in\B_n$ with the radius $\delta>0$, and the Kor\'{a}nyi region $\Gamma_{\gamma}(\xi)$ is defined by
$$\Gamma_{\gamma}(\xi)=\left\{z\in\mathbb{B}_n:|1-\langle z,\xi\rangle|<\frac{\gamma}{2}(1-|z|^2)\right\},$$
where $\xi\in\mathbb{S}_n$ and $\gamma>1$. It is known that for every $\delta>0$ and $\gamma>1$, there exists $\gamma'>1$ so that
\begin{equation*}
\bigcup_{z\in\Gamma_{\gamma}(\xi)}D(z,\delta)\subset\Gamma_{\gamma'}(\xi).
\end{equation*}
We will write $\widetilde{\Gamma}(\xi)$ to indicate this change of aperture, and the aperture of $\widetilde{\Gamma}(\xi)$ may change from one occurrence to the next.

For $z\in\B_n$, write $I(z)=\{\xi\in\mathbb{S}_n:z\in\Gamma(\xi)\}$. Then $\sigma(I(z))\asymp(1-|z|^2)^n$, and it follows from Fubini's theorem that
\begin{equation}\label{EqG}
\int_{\mathbb{B}_n}\varphi(z)d\nu(z)\asymp
              \int_{\mathbb{S}_n}\biggl(\int_{\Gamma(\xi)}\varphi(z)\frac{d\nu(z)}{(1-|z|^2)^n}\biggr)d\sigma(\xi),
\end{equation}
where $\varphi$ is a positive measurable function and $\nu$ is a finite positive measure on $\B_n$. An immediate consequence of this fact is that $\HT^p_{p,\alpha}=A^p_{n+\alpha}$ for all $0<p<\infty$ and $\alpha>-n-1$.

We will need the following inclusions between Hardy type tent spaces and Bergman spaces, which can be found in \cite[Lemma 3.1]{WZ21} and \cite[Lemma 2.4]{WZ22}.

\begin{lemma}\label{inclu}
Let $\alpha>-n-1$, $0<q<\infty$ and $0<p<t<\infty$. Then
$$\HT_{q,\alpha}^p\subset A^t_{\left(\frac{t}{p}-1\right)n-1+\frac{t(n+1+\alpha)}{q}}$$
and
$$A^p_{\left(\frac{p}{t}-1\right)n-1+\frac{p(n+1+\alpha)}{q}}\subset\HT_{q,\alpha}^t,$$
where both inclusions are bounded.
\end{lemma}

\subsection{Tent spaces of sequences}

We will need the discrete version of tent spaces. Before proceeding, we introduce a well-known result on decomposition of the unit ball $\mathbb{B}_n$. By \cite[Theorem 2.23]{Zhuball}, there exists a positive integer $N$ such that for any $0<\delta<1$ we can find a sequence $\{a_k\}\subset\mathbb{B}_n$ with the following properties:
\begin{enumerate}
	\item [(i)] $\mathbb{B}_n=\bigcup_kD(a_k,\delta)$;
	\item [(ii)] The sets $D(a_k,\delta/4)$ are mutually disjoint;
	\item [(iii)] Each point $z\in\mathbb{B}_n$ belongs to at most $N$ of the sets $D(a_k,4\delta)$.
\end{enumerate}
Any sequence $\{a_k\}$ satisfying the above conditions is called a $\delta$-lattice (in the Bergman metric).

Let $Z=\{a_k\}\subset\mathbb{B}_n$ be a $\delta$-lattice. For $0<p,q<\infty$, the tent space $T^p_q(Z)$ consists of those sequences $c=\{c_k\}\subset\C$ with
$$\|c\|^p_{T^p_q(Z)}:=\int_{\mathbb{S}_n}\left(\sum_{a_k\in\Gamma(\xi)}|c_k|^q\right)^{p/q}d\sigma(\xi)<\infty.$$
Analogously, the tent space $T^p_{\infty}(Z)$ consists of those sequences $c=\{c_k\}$ with
$$\|c\|^p_{T^p_{\infty}(Z)}:=\int_{\mathbb{S}_n}\left(\sup_{a_k\in\Gamma(\xi)}|c_k|\right)^pd\sigma(\xi)<\infty.$$
In order to define the tent space $T^{\infty}_q(Z)$, we write $Q(0)=\B_n$ and
$$Q(u):=\left\{z\in\B_n:\left|1-\left\langle z,\frac{u}{|u|}\right\rangle\right|<1-|u|^2\right\}$$
for $u\in\B_n\setminus\{0\}$. The tent space $T^{\infty}_q(Z)$ consists of sequences $c=\{c_k\}$ satisfying
$$\|c\|_{T^{\infty}_q(Z)}:=\mathrm{ess}\sup_{\xi\in\s_n}\left(\sup_{u\in\Gamma(\xi)}\frac{1}{(1-|u|^2)^n}
\sum_{a_k\in Q(u)}|c_k|^q(1-|a_k|^2)^n\right)^{1/q}<\infty.$$
By the definition of Carleson measures, we see that a sequence $c=\{c_k\}$ belongs to $T^{\infty}_q(Z)$ if and only if the measure $d\mu_c=\sum_k|c_k|^q(1-|a_k|^2)^n\delta_{a_k}$ is a Carleson measure, where $\delta_{a_k}$ is the Dirac point mass at $a_k$. Moreover,
$$\|c\|_{T^{\infty}_q(Z)}\asymp\|\mu_c\|^{1/q}_{CM}.$$

We will need the following lemma, which can be found in \cite[Lemma 14]{Per}.

\begin{lemma}\label{test}
Let $\alpha>-n-1$, $0<p,q<\infty$, $Z=\{a_k\}$ be a $\delta$-lattice, and $\theta>n\max\{1,q/p,1/p,1/q\}$. Then the operator $S^{\theta}_Z$, defined by
$$S^{\theta}_Z(\{\lambda_k\})(z):=\sum_{k=1}^{\infty}\lambda_k\frac{(1-|a_k|^2)^{\theta}}{(1-\langle z,a_k\rangle)^{\theta+\frac{n+1+\alpha}{q}}},\quad z\in\B_n,$$
is bounded from $T_q^p(Z)$ into $\HT_{q,\alpha}^p$.
\end{lemma}

The theory of dual and factorization of sequence tent spaces will play an essential role in the proof of Theorem \ref{main}, which we recall next. The following result about the dual of sequence tent spaces can be found in \cite{Ars,Jev,Lue91}.

\begin{theorem}\label{dual}
Let $Z=\{a_k\}$ be a $\delta$-lattice. If $1\leq p,q<\infty$ with $p+q>2$, then the dual of $T^p_q(Z)$ can be identified with $T^{p'}_{q'}(Z)$ under the pairing
$$\langle c,d\rangle_{T^2_2(Z)}:=\sum_kc_k\overline{d_k}(1-|a_k|^2)^n,
                              \quad c=\{c_k\}, \quad d=\{d_k\}.$$
Moreover, if $0<q<1<p<\infty$, then the dual of $T^p_q(Z)$ can be identified with $T^{p'}_{\infty}(Z)$ under the same pairing.
\end{theorem}

The following result concerning factorization of sequence tent spaces can be found in \cite[Proposition 6]{MPPW}.
\begin{theorem}\label{factor}
Let $0<p,q<\infty$ and $Z=\{a_k\}$ be a $\delta$-lattice. Suppose that $p_0\leq p_1,p_2\leq\infty$, $q_0\leq q_1,q_2\leq\infty$ and the following conditions hold:
\begin{enumerate}
	\item [(i)] $\frac{1}{p_j}+\frac{1}{q_j}>0$ for $j=0,1,2$;
	\item [(ii)] $\frac{1}{p_1}+\frac{1}{p_2}=\frac{1}{p_0}$ and $\frac{1}{q_1}+\frac{1}{q_2}=\frac{1}{q_0}$.
\end{enumerate}
Then
$$T^{p_0}_{q_0}(Z)=T^{p_1}_{q_1}(Z)\cdot T^{p_2}_{q_2}(Z).$$
That is, if $c\in T_{q_1}^{p_1}(Z)$ and $d\in T^{p_2}_{q_2}(Z)$, then $c\cdot d\in T^{p_0}_{q_0}(Z)$ with
$$\|c\cdot d\|_{T^{p_0}_{q_0}(Z)}\lesssim\|c\|_{T^{p_1}_{q_1}(Z)}\cdot\|d\|_{T^{p_2}_{q_2}(Z)};$$
and conversely, if $\lambda\in T^{p_0}_{q_0}(Z)$, then there exist sequences $c\in T_{q_1}^{p_1}(Z)$ and $d\in T^{p_2}_{q_2}(Z)$ such that $\lambda=c\cdot d$, and
$$\|c\|_{T^{p_1}_{q_1}(Z)}\cdot\|d\|_{T^{p_2}_{q_2}(Z)}\lesssim\|\lambda\|_{T^{p_0}_{q_0}(Z)}.$$
\end{theorem}

\subsection{Some integral estimates}

We next recall some useful integral estimates. The following Forelli--Rudin type estimate can be found in \cite[Lemma 2.5]{OF}.

\begin{lemma}\label{FRgeneral}
Let $s>-1$, $r,t>0$ and $r+t>s+n+1>r,t$. Then for $z,u\in\overline{\B_n}$, one has
$$\int_{\B_n}\frac{(1-|w|^2)^sdv(w)}{|1-\langle u,w\rangle|^r|1-\langle z,w\rangle|^t}
\lesssim\frac{1}{|1-\langle z,u\rangle|^{r+t-s-n-1}}.$$
\end{lemma}

We also need the Khinchine--Kahane inequalities regarding integration of random series. Let $\{r_k\}$ be the sequence of Rademacher functions on $[0,1]$ (see \cite[Appendix A]{Du} for example). For almost every $\tau\in[0,1]$, the sequence $\{r_k(\tau)\}$ consists of signs $\pm1$.  The following classical Khinchine's inequality can be found in \cite[Appendix A]{Du}, while Kahane’s inequality can be found in \cite[Lemma 5]{Lue93}.

\noindent{\bf Khinchine's inequality.} {\it Let $0<p<\infty$. Then for any sequence $\{c_k\}$ of complex numbers, one has
$$\int_0^1\left|\sum_{k}c_kr_k(\tau)\right|^pd\tau\asymp\left(\sum_{k}|c_k|^2\right)^{p/2}.$$}

\noindent{\bf Kahane's inequality.} {\it Let $X$ be a quasi-Banach space, and $0<p,q<\infty$. For any sequence $\{x_k\}\subset X$, one has
$$\left(\int_0^1\left\|\sum_kr_k(\tau)x_k\right\|^p_Xd\tau\right)^{1/p}\asymp
\left(\int_0^1\left\|\sum_kr_k(\tau)x_k\right\|^q_Xd\tau\right)^{1/q}.$$
Moreover, the implicit constants can be chosen to depend only on $p$ and $q$, and not on the quasi-Banach space $X$.}

\section{Proof of Theorem \ref{main}}\label{proof}

In this section, we are going to prove Theorem \ref{main}. It is clear that for any $f\in\h(\B_n)$, $\|f\|_{T^t_s(\mu)}=\|A_{\mu,s}(f)\|_{L^t(\s_n)}$. Hence the statements (a) and (b) of Theorem \ref{main} are equivalent. Moreover,
\begin{equation}\label{norm=}
\|I_d\|_{\HT_{q,\alpha}^p\to T^t_s(\mu)}=\|A_{\mu,s}\|_{\HT_{q,\alpha}^p\to L^t(\s_n)}.
\end{equation}
We now focus on the boundedness of $A_{\mu,s}:\HT_{q,\alpha}^p\to L^t(\s_n)$. Before proceeding, we display the following elementary lemma.

\begin{lemma}
Let $\alpha>-n-1$, $0<s,q<\infty$, and let $\mu$ be a positive Borel measure on $\B_n$. Let $Z=\{a_k\}$ be a $\delta$-lattice. Then for any $f\in\h(\B_n)$ and $\xi\in\s_n$,
\begin{equation}\label{ele1}
\left(A_{\mu,s}f(\xi)\right)^s
\lesssim\sum_{a_k\in\widetilde{\Gamma}(\xi)}\left(\int_{D(a_k,2\delta)}|f|^qdv_{\alpha}\right)^{s/q}
  \cdot\widehat{\mu}_{\delta}(a_k)(1-|a_k|^2)^{\frac{(q-s)(n+1+\alpha)}{q}},
\end{equation}
where the implicit constant is independent of $f$ and $\xi$. In particular, for any $r\in(0,1)$,
\begin{equation}\label{ele2}
\left(A_{\mu,s}f(\xi)\right)^s\lesssim\int_{\widetilde{\Gamma}(\xi)}|f(w)|^s\widehat{\mu}_{r}(w)dv_{\alpha}(w).
\end{equation}
\end{lemma}
\begin{proof}
By the sub-harmonic property of $|f|^q$, we obtain that
\begin{align*}
\left(A_{\mu,s}f(\xi)\right)^s
&\leq\sum_{a_k\in\widetilde{\Gamma}(\xi)}\int_{D(a_k,\delta)}|f(z)|^s\frac{d\mu(z)}{(1-|z|^2)^n}\\
&\lesssim\sum_{a_k\in\widetilde{\Gamma}(\xi)}\int_{D(a_k,\delta)}
  \left(\frac{1}{(1-|z|^2)^{n+1+\alpha}}\int_{D(z,\delta)}|f|^qdv_{\alpha}\right)^{s/q}\frac{d\mu(z)}{(1-|z|^2)^n}\\
&\lesssim\sum_{a_k\in\widetilde{\Gamma}(\xi)}\left(\int_{D(a_k,2\delta)}|f|^qdv_{\alpha}\right)^{s/q}
  \cdot\frac{\mu(D(a_k,\delta))}{(1-|a_k|^2)^{n+\frac{s(n+1+\alpha)}{q}}}\\
&=\sum_{a_k\in\widetilde{\Gamma}(\xi)}\left(\int_{D(a_k,2\delta)}|f|^qdv_{\alpha}\right)^{s/q}
  \cdot\widehat{\mu}_{\delta}(a_k)(1-|a_k|^2)^{\frac{(q-s)(n+1+\alpha)}{q}},
\end{align*}
which establishes \eqref{ele1}. Noting that $D(a_k,\delta)\subset D(w,3\delta)$ whenever $w\in D(a_k,2\delta)$, we may choose $q=s$ and $\delta=r/3$ in \eqref{ele1} to obtain that
\begin{align*}
\left(A_{\mu,s}f(\xi)\right)^s
&\lesssim\sum_{a_k\in\widetilde{\Gamma}(\xi)}\int_{D(a_k,2\delta)}|f(w)|^s\widehat{\mu}_{3\delta}(w)dv_{\alpha}(w)\\
&\lesssim\int_{\widetilde{\Gamma}(\xi)}|f(w)|^s\widehat{\mu}_{r}(w)dv_{\alpha}(w),
\end{align*}
which gives \eqref{ele2}.
\end{proof}

The following lemma establishes a family of test functions in $\HT^p_{q,\alpha}$.

\begin{lemma}\label{te}
Let $\alpha>-n-1$ and $0<p,q,\theta<\infty$. For $a\in\B_n$, define
$$f_a(z)=\frac{(1-|a|^2)^{\theta}}{(1-\langle z,a \rangle)^{\theta+\frac{n+1+\alpha}{q}+\frac{n}{p}}},\quad z\in\B_n.$$
Then $f_a\in\HT^p_{q,\alpha}$, and $\|f_a\|_{\HT^p_{q,\alpha}}\lesssim1.$
\end{lemma}
\begin{proof}
It is clear that $f_a$ is bounded on $\B_n$. Therefore, $f_a\in\HT^p_{q,\alpha}$. Fix $\kappa>\max\{n,\theta q+nq/p\}$. Noting that for any $\xi\in\s_n$, $|1-\langle z,\xi\rangle|\asymp1-|z|^2$ if $z\in\Gamma(\xi)$, we can apply Lemma \ref{FRgeneral} and the classical Forelli--Rudin estimate (see \cite[Theorem 1.12]{Zhuball}) to obtain that
\begin{align*}
\|f_a\|^p_{\HT^p_{q,\alpha}}
&=\int_{\s_n}\left(\int_{\Gamma(\xi)}\frac{(1-|a|^2)^{\theta q}dv_{\alpha}(z)}
  {|1-\langle z,a\rangle|^{\theta q+n+1+\alpha+\frac{nq}{p}}}\right)^{p/q}d\sigma(\xi)\\
&\asymp\int_{\s_n}\left(\int_{\Gamma(\xi)}\frac{(1-|z|^2)^{\kappa}}{|1-\langle z,\xi\rangle|^{\kappa}}
  \frac{(1-|a|^2)^{\theta q}dv_{\alpha}(z)}
  {|1-\langle z,a\rangle|^{\theta q+n+1+\alpha+\frac{nq}{p}}}\right)^{p/q}d\sigma(\xi)\\
&\lesssim\int_{\s_n}\frac{(1-|a|^2)^{\theta p}}{|1-\langle a,\xi\rangle|^{\theta p+n}}d\sigma(\xi)\lesssim1,
\end{align*}
which finishes the proof.
\end{proof}

We will separate the boundedness of $A_{\mu,s}:\HT_{q,\alpha}^p\to L^t(\s_n)$ into four results according to what case is considered. The following is the first case with the corresponding estimate for the norm of $A_{\mu,s}$.

\begin{theorem}\label{case1}
Let $\alpha>-n-1$, $p<t$, or $p=t$ and $q\leq s$. Then $A_{\mu,s}:\HT_{q,\alpha}^p\to L^t(\s_n)$ is bounded if and only if
$$G_{\mu}(z):=\widehat{\mu}_r(z)(1-|z|^2)^{\frac{(q-s)(n+1+\alpha)}{q}+ns\left(\frac{1}{t}-\frac{1}{p}\right)}$$
is bounded on $\B_n$. Moreover,
$$\|A_{\mu,s}\|_{\HT_{q,\alpha}^p\to L^t(\s_n)}\asymp\sup_{z\in\B_n}G_{\mu}(z)^{1/s}.$$
\end{theorem}
\begin{proof}
Suppose first that $A_{\mu,s}:\HT_{q,\alpha}^p\to L^t(\s_n)$ is bounded. For $a\in\B_n$, define $f_a$ as in Lemma \ref{te}. It is well-known that
$$|f_a(z)|\asymp(1-|a|^2)^{-\frac{n+1+\alpha}{q}-\frac{n}{p}},\quad z\in D(a,r).$$
Since $A_{\mu,s}:\HT_{q,\alpha}^p\to L^t(\s_n)$ is bounded, we have
\begin{align*}
\|A_{\mu,s}\|^t_{\HT_{q,\alpha}^p\to L^t(\s_n)}
&\gtrsim\|A_{\mu,s}f_a\|^t_{L^t(\s_n)}\\
&\asymp\int_{\s_n}\left(\int_{\widetilde{\Gamma}(\xi)}|f_a(z)|^s\frac{d\mu(z)}{(1-|z|^2)^n}\right)^{t/s}d\sigma(\xi)\\
&\geq\int_{T(a)}\left(\int_{D(a,r)}|f_a(z)|^s\frac{d\mu(z)}{(1-|z|^2)^n}\right)^{t/s}d\sigma(\xi)\\
&\asymp\left(\mu(D(a,r))(1-|a|^2)^{-\frac{s(n+1+\alpha)}{q}-\frac{sn}{p}-n+\frac{ns}{t}}\right)^{t/s}\\
&=G_{\mu}(a)^{t/s}.
\end{align*}
The arbitrariness of $a\in\B_n$ implies that $G_{\mu}$ is bounded on $\B_n$, and $$\sup_{a\in\B_n}G_{\mu}(a)\lesssim\|A_{\mu,s}\|^s_{\HT_{q,\alpha}^p\to L^t(\s_n)}.$$

Conversely, suppose that $G_{\mu}$ is bounded on $\B_n$. Fix $f\in\HT_{q,\alpha}^p$. Then for any $\xi\in\s_n$, \eqref{ele2} yields that
\begin{equation}\label{xi}
\left(A_{\mu,s}f(\xi)\right)^s\lesssim\int_{\widetilde{\Gamma}(\xi)}|f(z)|^s\widehat{\mu}_{r}(z)dv_{\alpha}(z)
\leq\sup_{\B_n}G_{\mu}\cdot\int_{\widetilde{\Gamma}(\xi)}|f(z)|^sdv_{\beta}(z),
\end{equation}
where
$$\beta=\frac{(s-q)(n+1+\alpha)}{q}+ns\left(\frac{1}{p}-\frac{1}{t}\right)+\alpha.$$
We first consider the case of $p=t$ and $q\leq s$. Let $Z=\{a_k\}$ be a $\delta$-lattice. Then by the sub-harmonic property of $|f|^q$ and H\"{o}lder's inequality,
\begin{align}\label{Hol}
&\nonumber\int_{\widetilde{\Gamma}(\xi)}|f(z)|^sdv_{\beta}(z)\\
&\nonumber\ \lesssim\sum_{a_k\in\widetilde{\Gamma}(\xi)}\int_{D(a_k,\delta)}
    \left(\int_{D(z,\delta)}|f|^qdv_{\alpha}\right)^{s/q}(1-|z|^2)^{\beta-\frac{s(n+1+\alpha)}{q}}dv(z)\\
&\nonumber\ \lesssim\sum_{a_k\in\widetilde{\Gamma}(\xi)}\left(\int_{D(a_k,2\delta)}|f|^qdv_{\alpha}\right)^{s/q}
    \leq\left(\sum_{a_k\in\widetilde{\Gamma}(\xi)}\int_{D(a_k,2\delta)}|f|^qdv_{\alpha}\right)^{s/q}\\
&\ \lesssim\left(\int_{\widetilde{\Gamma}(\xi)}|f|^qdv_{\alpha}\right)^{s/q}.
\end{align}
Inserting this into \eqref{xi}, we obtain that
\begin{align*}
\int_{\s_n}\left(A_{\mu,s}f(\xi)\right)^td\sigma(\xi)
&\lesssim\sup_{\B_n}G_{\mu}^{t/s}\cdot\int_{\s_n}\left(\int_{\widetilde{\Gamma}(\xi)}|f|^qdv_{\alpha}\right)^{t/q}d\sigma(\xi)\\
&\asymp\sup_{\B_n}G_{\mu}^{t/s}\cdot\|f\|^t_{\HT_{q,\alpha}^p}.
\end{align*}
Hence $A_{\mu,s}:\HT_{q,\alpha}^p\to L^t(\s_n)$ is bounded, and
$$\|A_{\mu,s}\|_{\HT_{q,\alpha}^p\to L^t(\s_n)}\lesssim\sup_{\B_n}G_{\mu}^{1/s}.$$
We next consider the case of $p<t$. Choose $p<\hat{p}<t$. Then it is clear that $\beta>-n-1$, and
$$\left(\frac{\hat{p}}{t}-1\right)n-1+\frac{\hat{p}(n+1+\beta)}{s}=\left(\frac{\hat{p}}{p}-1\right)n-1+\frac{\hat{p}(n+1+\alpha)}{q}.$$
Consequently, it follows from Lemma \ref{inclu} that
\begin{equation}\label{inclus}
\HT_{q,\alpha}^p\subset A^{\hat{p}}_{\left(\frac{\hat{p}}{p}-1\right)n-1+\frac{\hat{p}(n+1+\alpha)}{q}}\subset\HT_{s,\beta}^t
\end{equation}
with bounded inclusions. Since $f\in\HT^p_{q,\alpha}$, we may apply \eqref{xi} and \eqref{inclus} to obtain that
\begin{align*}
\int_{\s_n}\left(A_{\mu,s}f(\xi)\right)^td\sigma(\xi)
&\lesssim\sup_{\B_n}G_{\mu}^{t/s}\cdot
    \int_{\s_n}\left(\int_{\widetilde{\Gamma}(\xi)}|f|^sdv_{\beta}\right)^{t/s}d\sigma(\xi)\\
&\asymp\sup_{\B_n}G_{\mu}^{t/s}\cdot\|f\|^t_{\HT_{s,\beta}^t}\\
&\lesssim\sup_{\B_n}G_{\mu}^{t/s}\cdot\|f\|^t_{\HT_{q,\alpha}^p},
\end{align*}
which implies that $A_{\mu,s}:\HT_{q,\alpha}^p\to L^t(\s_n)$ is bounded, and
$$\|A_{\mu,s}\|_{\HT_{q,\alpha}^p\to L^t(\s_n)}\lesssim\sup_{\B_n}G_{\mu}^{1/s}.$$
The proof is complete.
\end{proof}

In the rest part of this section, we always let $Z=\{a_k\}$ be an $\frac{r}{2}$-lattice, and for $\delta>0$, define the sequence $\eta^{(\delta)}=\{\eta^{(\delta)}_k\}$ by
$$\eta^{(\delta)}_k=\widehat{\mu}_{\delta}(a_k)(1-|a_k|^2)^{\frac{(q-s)(n+1+\alpha)}{q}}.$$
To handle other cases of the boundedness of $A_{\mu,s}:\HT_{q,\alpha}^p\to L^t(\s_n)$, we need the following proposition.

\begin{proposition}\label{nece}
Let $\alpha>-n-1$ and $0<p,q,s,t<\infty$. Suppose that $A_{\mu,s}:\HT_{q,\alpha}^p\to L^t(\s_n)$ is bounded. Then for any $\lambda=\{\lambda_k\}\in T^p_q(Z)$,
$$\int_{\s_n}\left(\sum_{a_k\in\Gamma(\xi)}|\lambda_k|^s\eta^{(2r)}_k\right)^{t/s}d\sigma(\xi)\lesssim
\|A_{\mu,s}\|^{t}_{\HT_{q,\alpha}^p\to L^t(\s_n)}\cdot\|\lambda\|^t_{T_q^p(Z)}.$$
\end{proposition}
\begin{proof}
Fix $\theta>n\max\{1,q/p,1/p,1/q\}$. For any $\lambda=\{\lambda_k\}\in T_q^p(Z)$, define
$$F_{\tau}(z)=\sum_{k=1}^{\infty}\lambda_kr_k(\tau)f_k(z),\quad z\in\B_n,\quad \tau\in[0,1],$$
where $\{r_k\}$ is the sequence of Rademacher functions on $[0,1]$, and
$$f_k(z)=\frac{(1-|a_k|^2)^{\theta}}{(1-\langle z,a_k\rangle)^{\theta+\frac{n+1+\alpha}{q}}}.$$
Then by Lemma \ref{test}, $F_{\tau}\in\HT^p_{q,\alpha}$, and $\|F_{\tau}\|_{\HT_{q,\alpha}^p}\lesssim\|\lambda\|_{T_q^p(Z)}$ for any $\tau\in[0,1]$. The boundedness of $A_{\mu,s}:\HT_{q,\alpha}^p\to L^t(\s_n)$ then gives that
\begin{align}\label{inte}
\nonumber\left\|A_{\mu,s}F_{\tau}\right\|^t_{L^t(\s_n)}=
\int_{\s_n}\left(\int_{\Gamma(\xi)}\left|\sum_{k=1}^{\infty}\lambda_kr_k(\tau)f_k(z)\right|^s\frac{d\mu(z)}{(1-|z|^2)^n}\right)^{t/s}
    d\sigma(\xi)\\
\lesssim\|A_{\mu,s}\|^{t}_{\HT_{q,\alpha}^p\to L^t(\s_n)}\cdot\|\lambda\|^t_{T_q^p(Z)}.
\end{align}
Applying Fubini's theorem, Kahane and Khinchine’s inequalities, we obtain that
\begin{align*}
&\int_0^1\left\|A_{\mu,s}F_{\tau}\right\|^t_{L^t(\s_n)}d\tau\\
&\ =\int_{\s_n}\int_0^1\left(\int_{\Gamma(\xi)}\left|\sum_{k=1}^{\infty}\lambda_kr_k(\tau)f_k(z)\right|^s
    \frac{d\mu(z)}{(1-|z|^2)^n}\right)^{t/s}d\tau d\sigma(\xi)\\
&\ \asymp\int_{\s_n}\left(\int_0^1\int_{\Gamma(\xi)}\left|\sum_{k=1}^{\infty}\lambda_kr_k(\tau)f_k(z)\right|^s
    \frac{d\mu(z)}{(1-|z|^2)^n}d\tau\right)^{t/s}d\sigma(\xi)\\
&\ \asymp\int_{\s_n}\left(\int_{\Gamma(\xi)}\left(\sum_{k=1}^{\infty}|\lambda_k|^2|f_k(z)|^2\right)^{s/2}
    \frac{d\mu(z)}{(1-|z|^2)^n}\right)^{t/s}d\sigma(\xi).
\end{align*}
Noting that each point $z\in\B_n$ belongs to at most $N$ of the sets $D(a_k,2r)$, we get
\begin{align*}
\left(\sum_{k=1}^{\infty}|\lambda_k|^2|f_k(z)|^2\right)^{s/2}
&\geq\left(\sum_{k=1}^{\infty}|\lambda_k|^2\frac{(1-|a_k|^2)^{2\theta}\chi_{D(a_k,2r)}(z)}
    {|1-\langle z,a_k\rangle|^{2\theta+\frac{2(n+1+\alpha)}{q}}}\right)^{s/2}\\
&\gtrsim\sum_{k=1}^{\infty}|\lambda_k|^s\frac{\chi_{D(a_k,2r)}(z)}{(1-|a_k|^2)^{\frac{s(n+1+\alpha)}{q}}}.
\end{align*}
Combining this with the fact that there exists $\gamma>1$ such that $D(z,2r)\subset\Gamma(\xi)$ if $z\in\Gamma_{\gamma}(\xi)$, we arrive at
\begin{align*}
&\int_0^1\left\|A_{\mu,s}F_{\tau}\right\|^t_{L^t(\s_n)}d\tau\\
&\ \gtrsim\int_{\s_n}\left(\int_{\Gamma(\xi)}
    \sum_{k=1}^{\infty}|\lambda_k|^s\frac{\chi_{D(a_k,2r)}(z)}{(1-|a_k|^2)^{\frac{s(n+1+\alpha)}{q}}}
    \frac{d\mu(z)}{(1-|z|^2)^n}\right)^{t/s}d\sigma(\xi)\\
&\ \gtrsim\int_{\s_n}\left(\sum_{a_k\in\Gamma_{\gamma}(\xi)}|\lambda_k|^s
    \frac{\mu(D(a_k,2r))}{(1-|a_k|^2)^{n+\frac{s(n+1+\alpha)}{q}}}\right)^{t/s}d\sigma(\xi)\\
&\ \asymp\int_{\s_n}\left(\sum_{a_k\in\Gamma(\xi)}|\lambda_k|^s\eta^{(2r)}_k\right)^{t/s}d\sigma(\xi),
\end{align*}
which, in conjunction with \eqref{inte}, completes the proof.
\end{proof}

We also need the following discretization lemma to deal with the case of $p=t$ and $q>s$.

\begin{lemma}\label{disCM}
Let $\alpha>-n-1$ and $0<s<q<\infty$. Consider the following condition:
\begin{enumerate}
	\item [(a)] $\eta^{(2r)}\in T^{\infty}_{\frac{q}{q-s}}(Z)$;
	\item [(b)] The measure
	$$d\nu_{\mu}(z):=\widehat{\mu}_r(z)^{\frac{q}{q-s}}dv_{\alpha+n}(z)$$
	is a Carleson measure;
	\item [(c)] $\eta^{(r/2)}\in T^{\infty}_{\frac{q}{q-s}}(Z)$.
\end{enumerate}
Then we have the implications (a)$\Longrightarrow$(b)$\Longrightarrow$(c). Moreover,
$$\left\|\eta^{(r/2)}\right\|_{T^{\infty}_{\frac{q}{q-s}}(Z)}\lesssim\|\nu_{\mu}\|^{\frac{q-s}{q}}_{CM}\lesssim
\left\|\eta^{(2r)}\right\|_{T^{\infty}_{\frac{q}{q-s}}(Z)}.$$
\end{lemma}
\begin{proof}
We only prove the implication (a)$\Longrightarrow$(b). The other is similar. Suppose $\eta^{(2r)}\in T^{\infty}_{\frac{q}{q-s}}(Z)$. Then the measure
$$\sum_{k=1}^{\infty}\left(\eta^{(2r)}_k\right)^{\frac{q}{q-s}}(1-|a_k|^2)^n\delta_{a_k}$$
is a Carleson measure. Hence for any $\theta>0$ and $a\in\B_n$, by the characterization \eqref{CM},
\begin{align*}
\int_{\B_n}\frac{(1-|a|^2)^{\theta}}{|1-\langle z,a\rangle|^{n+\theta}}d\nu_{\mu}(z)
&\leq\sum_{k=1}^{\infty}\int_{D(a_k,r)}\frac{(1-|a|^2)^{\theta}}{|1-\langle z,a\rangle|^{n+\theta}}
    \widehat{\mu}_r(z)^{\frac{q}{q-s}}dv_{\alpha+n}(z)\\
&\lesssim\sum_{k=1}^{\infty}\frac{(1-|a|^2)^{\theta}}{|1-\langle a_k,a\rangle|^{n+\theta}}
    \widehat{\mu}_{2r}(a_k)^{\frac{q}{q-s}}(1-|a_k|^2)^{2n+1+\alpha}\\
&=\sum_{k=1}^{\infty}\frac{(1-|a|^2)^{\theta}}{|1-\langle a_k,a\rangle|^{n+\theta}}\left(\eta^{(2r)}_k\right)^{\frac{q}{q-s}}(1-|a_k|^2)^n\\
&\lesssim\left\|\eta^{(2r)}\right\|^{\frac{q}{q-s}}_{T^{\infty}_{\frac{q}{q-s}}(Z)}.
\end{align*}
Therefore, $\nu_{\mu}$ is a Carleson measure, and $\|\nu_{\mu}\|_{CM}\lesssim
\left\|\eta^{(2r)}\right\|^{\frac{q}{q-s}}_{T^{\infty}_{\frac{q}{q-s}}(Z)}$.
\end{proof}

\begin{theorem}\label{case2}
Let $\alpha>-n-1$, $p=t$ and $q>s$. Then $A_{\mu,s}:\HT_{q,\alpha}^p\to L^t(\s_n)$ is bounded if and only if
$$d\nu_{\mu}(z):=\widehat{\mu}_r(z)^{\frac{q}{q-s}}dv_{\alpha+n}(z)$$
is a Carleson measure. Moreover,
$$\|A_{\mu,s}\|_{\HT_{q,\alpha}^p\to L^t(\s_n)}\asymp\|\nu_{\mu}\|^{\frac{q-s}{qs}}_{CM}.$$
\end{theorem}
\begin{proof}
We first consider the necessity. By Lemma \ref{disCM}, it suffices to show $\eta^{(2r)}\in T^{\infty}_{\frac{q}{q-s}}(Z)$, which is equivalent to
$$\left(\eta^{(2r)}\right)^{1/\sigma}:=\left\{\left(\eta_k^{(2r)}\right)^{1/\sigma}\right\}\in T^{\infty}_{\frac{q\sigma}{q-s}}(Z),$$
where $\sigma>1$ is large enough such that $p\sigma>s$.
We will accomplish this by the dual and factorization of sequence tent spaces. In fact, by Theorems \ref{dual} and \ref{factor}, we have
$$T^{\infty}_{\frac{q\sigma}{q-s}}(Z)=\left(T^1_{\left(\frac{q\sigma}{q-s}\right)'}(Z)\right)^*
=\left(T^{\frac{p\sigma}{p\sigma-s}}_{\sigma'}(Z)\cdot T^{p\sigma/s}_{q\sigma/s}(Z)\right)^*.$$
Take any $c=\{c_k\}\in T^1_{\left(\frac{q\sigma}{q-s}\right)'}(Z)$, and factor it as suggested as
$$c_k=d_k\lambda^{s/\sigma}_{k},\quad d=\{d_k\}\in T^{\frac{p\sigma}{p\sigma-s}}_{\sigma'}(Z),\quad \lambda=\{\lambda_k\}\in T^{p}_{q}(Z),$$
with
$$\|d\|_{T^{\frac{p\sigma}{p\sigma-s}}_{\sigma'}(Z)}\cdot\|\lambda\|^{s/\sigma}_{T^{p}_{q}(Z)}\lesssim \|c\|_{T^1_{\left(\frac{q\sigma}{q-s}\right)'}(Z)}.$$
Using this factorization, the estimate \eqref{EqG}, H\"{o}lder's inequality and Proposition \ref{nece}, we obtain that
\begin{align*}
\left|\left\langle c,\left(\eta^{(2r)}\right)^{1/\sigma}\right\rangle_{T^2_2(Z)}\right|
&\lesssim\int_{\s_n}\left(\sum_{a_k\in\Gamma(\xi)}|d_k||\lambda_k|^{s/\sigma}\left(\eta_k^{(2r)}\right)^{1/\sigma}\right)d\sigma(\xi)\\
&\leq\int_{\s_n}\left(\sum_{a_k\in\Gamma(\xi)}|d_k|^{\sigma'}\right)^{1/{\sigma'}}
    \left(\sum_{a_k\in\Gamma(\xi)}|\lambda_k|^{s}\eta_k^{(2r)}\right)^{1/\sigma}d\sigma(\xi)\\
&\leq\|d\|_{T^{\frac{p\sigma}{p\sigma-s}}_{\sigma'}(Z)}\cdot
    \left(\int_{\s_n}\left(\sum_{a_k\in\Gamma(\xi)}|\lambda_k|^{s}\eta_k^{(2r)}\right)^{t/s}d\sigma(\xi)\right)^{\frac{s}{t\sigma}}\\
&\lesssim\|d\|_{T^{\frac{p\sigma}{p\sigma-s}}_{\sigma'}(Z)}\cdot\|A_{\mu,s}\|^{s/\sigma}_{\HT_{q,\alpha}^p\to L^t(\s_n)}\cdot
    \|\lambda\|^{s/\sigma}_{T_q^p(Z)}\\
&\lesssim\|A_{\mu,s}\|^{s/\sigma}_{\HT_{q,\alpha}^p\to L^t(\s_n)}\cdot\|c\|_{T^1_{\left(\frac{q\sigma}{q-s}\right)'}(Z)}.
\end{align*}
By the duality of sequence tent spaces, we obtain that $\left(\eta^{(2r)}\right)^{1/\sigma}\in T^{\infty}_{\frac{q\sigma}{q-s}}(Z)$ with
$$\left\|\left(\eta^{(2r)}\right)^{1/\sigma}\right\|_{T^{\infty}_{\frac{q\sigma}{q-s}}(Z)}\lesssim
\|A_{\mu,s}\|^{s/\sigma}_{\HT_{q,\alpha}^p\to L^t(\s_n)}.$$
Therefore, Lemma \ref{disCM} implies that $\nu_{\mu}$ is a Carleson measure, and
$$\|\nu_{\mu}\|_{CM}\lesssim\left\|\eta^{(2r)}\right\|^{\frac{q}{q-s}}_{T^{\infty}_{\frac{q}{q-s}}(Z)}
=\left\|\left(\eta^{(2r)}\right)^{1/\sigma}\right\|^{\frac{q\sigma}{q-s}}_{T^{\infty}_{\frac{q\sigma}{q-s}}(Z)}\lesssim
\|A_{\mu,s}\|^{\frac{qs}{q-s}}_{\HT_{q,\alpha}^p\to L^t(\s_n)}.$$

We next look for the sufficiency. For any $f\in\HT^p_{q,\alpha}$, applying \eqref{ele1} with $\delta=r/2$ yields that
$$\left(A_{\mu,s}f(\xi)\right)^s
\lesssim\sum_{a_k\in\widetilde{\Gamma}(\xi)}\left(\int_{D(a_k,r)}|f|^qdv_{\alpha}\right)^{s/q}
\cdot\eta^{(r/2)}_k.$$
Since $\nu_{\mu}$ is a Carleson measure, Lemma \ref{disCM} implies that $\eta^{(r/2)}\in T^{\infty}_{\frac{q}{q-s}}(Z)$ with
$$\left\|\eta^{(r/2)}\right\|_{T^{\infty}_{\frac{q}{q-s}}(Z)}\lesssim\|\nu_{\mu}\|^{\frac{q-s}{q}}_{CM}.$$
On the other hand, it is easy to verify that
$$\left\{\left(\int_{D(a_k,r)}|f|^qdv_{\alpha}\right)^{s/q}\right\}\in T^{p/s}_{q/s}(Z),$$
and
$$\left\|\left\{\left(\int_{D(a_k,r)}|f|^qdv_{\alpha}\right)^{s/q}\right\}\right\|_{T^{p/s}_{q/s}(Z)}\lesssim\|f\|^s_{\HT_{q,\alpha}^p}.$$
Bearing in mind that $p=t$, we may apply Theorem \ref{factor} to obtain that
$$\left\{\left(\int_{D(a_k,r)}|f|^qdv_{\alpha}\right)^{s/q}\cdot\eta^{(r/2)}_k\right\}
\in T^{p/s}_{q/s}(Z)\cdot T^{\infty}_{\frac{q}{q-s}}(Z)=T^{t/s}_{1}(Z)$$
with
\begin{align*}
&\left\|\left\{\left(\int_{D(a_k,r)}|f|^qdv_{\alpha}\right)^{s/q}\cdot\eta^{(r/2)}_k\right\}\right\|_{T^{t/s}_{1}(Z)}\\
&\ \lesssim\left\|\left\{\left(\int_{D(a_k,r)}|f|^qdv_{\alpha}\right)^{s/q}\right\}\right\|_{T^{p/s}_{q/s}(Z)}\cdot
    \left\|\eta^{(r/2)}\right\|_{T^{\infty}_{\frac{q}{q-s}}(Z)}\\
&\ \lesssim\|\nu_{\mu}\|^{\frac{q-s}{q}}_{CM}\cdot\|f\|^s_{\HT_{q,\alpha}^p}.
\end{align*}
Consequently,
\begin{align*}
\int_{\s_n}\left(A_{\mu,s}f(\xi)\right)^td\sigma(\xi)
&\lesssim\int_{\s_n}\left(\sum_{a_k\in\widetilde{\Gamma}(\xi)}\left(\int_{D(a_k,r)}|f|^qdv_{\alpha}\right)^{s/q}
    \cdot\eta^{(r/2)}_k\right)^{t/s}d\sigma(\xi)\\
&\asymp\left\|\left\{\left(\int_{D(a_k,r)}|f|^qdv_{\alpha}\right)^{s/q}\cdot\eta^{(r/2)}_k\right\}\right\|^{t/s}_{T^{t/s}_{1}(Z)}\\
&\lesssim\|\nu_{\mu}\|^{\frac{t(q-s)}{qs}}_{CM}\cdot\|f\|^t_{\HT_{q,\alpha}^p}.
\end{align*}
Therefore, $A_{\mu,s}:\HT_{q,\alpha}^p\to L^t(\s_n)$ is bounded, and
$$\|A_{\mu,s}\|_{\HT_{q,\alpha}^p\to L^t(\s_n)}\lesssim\|\nu_{\mu}\|^{\frac{q-s}{qs}}_{CM},$$
which finishes the proof.
\end{proof}

To handle the case $p>t$, we need the following discretization lemma, which can be found in \cite[Lemmas 2.1 and 2.2]{LvPW}.

\begin{lemma}\label{dis}
Let $\beta\in\mathbb{R}$ and $0<p,q<\infty$. Then
\begin{align*}
\int_{\s_n}\left(\int_{\Gamma(\xi)}\widehat{\mu}_r(z)^pdv_{\beta}(z)\right)^qd\sigma(\xi)
\lesssim\int_{\s_n}\left(\sum_{a_k\in\Gamma(\xi)}\widehat{\mu}_{2r}(a_k)^p(1-|a_k|^2)^{n+1+\beta}\right)^qd\sigma(\xi)
\end{align*}
and
\begin{align*}
\int_{\s_n}\sup_{z\in\Gamma(\xi)}\widehat{\mu}_r(z)^p(1-|z|^2)^{\beta}d\sigma(\xi)
\lesssim\int_{\s_n}\sup_{a_k\in\Gamma(\xi)}\widehat{\mu}_{2r}(a_k)^p(1-|a_k|^2)^{\beta}d\sigma(\xi).
\end{align*}
\end{lemma}

\begin{theorem}\label{case3}
Let $\alpha>-n-1$, $p>t$ and $q>s$. Then $A_{\mu,s}:\HT_{q,\alpha}^p\to L^t(\s_n)$ is bounded if and only if
$$U_{\mu}(\xi):=\left(\int_{\Gamma(\xi)}\widehat{\mu}_r(z)^{\frac{q}{q-s}}dv_{\alpha}(z)\right)^{\frac{q-s}{q}}$$
belongs to $L^{\frac{pt}{s(p-t)}}(\s_n)$. Moreover,
$$\|A_{\mu,s}\|_{\HT_{q,\alpha}^p\to L^t(\s_n)}\asymp\|U_{\mu}\|^{1/s}_{L^{\frac{pt}{s(p-t)}}(\s_n)}.$$
\end{theorem}
\begin{proof}
Suppose first that $A_{\mu,s}:\HT_{q,\alpha}^p\to L^t(\s_n)$ is bounded. We are going to establish 
$\eta^{(2r)}\in T^{\frac{pt}{s(p-t)}}_{\frac{q}{q-s}}(Z)$, which is equivalent to $\left(\eta^{(2r)}\right)^{1/\sigma}\in T^{\frac{pt\sigma}{s(p-t)}}_{\frac{q\sigma}{q-s}}(Z)$, where $\sigma>1$ is large enough such that $t\sigma>s$. By Theorems \ref{dual} and \ref{factor}, we have
$$\left(T^{\frac{pt\sigma}{s(p-t)}}_{\frac{q\sigma}{q-s}}(Z)\right)^*
=T^{\left(\frac{pt\sigma}{s(p-t)}\right)'}_{\left(\frac{q\sigma}{q-s}\right)'}(Z)
=T^{\frac{t\sigma}{t\sigma-s}}_{\sigma'}(Z)\cdot T^{p\sigma/s}_{q\sigma/s}(Z).$$
Take any $c=\{c_k\}\in T^{\left(\frac{pt\sigma}{s(p-t)}\right)'}_{\left(\frac{q\sigma}{q-s}\right)'}(Z)$, and factor it as $c_k=d_k\lambda_k^{s/\sigma}$, with $d=\{d_k\}\in T^{\frac{t\sigma}{t\sigma-s}}_{\sigma'}(Z)$ and $\lambda=\{\lambda_k\}\in T_q^p(Z)$. Applying \eqref{EqG}, H\"{o}lder's inequality and Proposition \ref{nece} as in the proof of Theorem \ref{case2}, we obtain that
\begin{align*}
\left|\left\langle c,\left(\eta^{(2r)}\right)^{1/\sigma}\right\rangle_{T^2_2(Z)}\right|
&\lesssim\|d\|_{T^{\frac{t\sigma}{t\sigma-s}}_{\sigma'}(Z)}\cdot
    \left(\int_{\s_n}\left(\sum_{a_k\in\Gamma(\xi)}|\lambda_k|^{s}\eta_k^{(2r)}\right)^{t/s}d\sigma(\xi)\right)^{\frac{s}{t\sigma}}\\
&\lesssim\|d\|_{T^{\frac{t\sigma}{t\sigma-s}}_{\sigma'}(Z)}\cdot\|A_{\mu,s}\|^{s/\sigma}_{\HT_{q,\alpha}^p\to L^t(\s_n)}\cdot
    \|\lambda\|^{s/\sigma}_{T_q^p(Z)}\\
&\lesssim\|A_{\mu,s}\|^{s/\sigma}_{\HT_{q,\alpha}^p\to L^t(\s_n)}
    \cdot\|c\|_{T^{\left(\frac{pt\sigma}{s(p-t)}\right)'}_{\left(\frac{q\sigma}{q-s}\right)'}(Z)}.
\end{align*}
By duality we get that $\left(\eta^{(2r)}\right)^{1/\sigma}\in T^{\frac{pt\sigma}{s(p-t)}}_{\frac{q\sigma}{q-s}}(Z)$ with
$$\left\|\left(\eta^{(2r)}\right)^{1/\sigma}\right\|_{T^{\frac{pt\sigma}{s(p-t)}}_{\frac{q\sigma}{q-s}}(Z)}
\lesssim\|A_{\mu,s}\|^{s/\sigma}_{\HT_{q,\alpha}^p\to L^t(\s_n)},$$
which, in conjunction with Lemma \ref{dis}, implies that $U_{\mu}\in L^{\frac{pt}{s(p-t)}}(\s_n)$, and
$$\|U_{\mu}\|_{L^{\frac{pt}{s(p-t)}}(\s_n)}\lesssim\left\|\eta^{(2r)}\right\|_{T^{\frac{pt}{s(p-t)}}_{\frac{q}{q-s}}(Z)}
=\left\|\left(\eta^{(2r)}\right)^{1/\sigma}\right\|^{\sigma}_{T^{\frac{pt\sigma}{s(p-t)}}_{\frac{q\sigma}{q-s}}(Z)}
\lesssim\|A_{\mu,s}\|^{s}_{\HT_{q,\alpha}^p\to L^t(\s_n)}.$$

Conversely, suppose that $U_{\mu}\in L^{\frac{pt}{s(p-t)}}(\s_n)$. Note that $p/t>1$ and $q/s>1$. For any $f\in\HT_{q,\alpha}^p$, the estimate \eqref{ele2} together with H\"{o}lder's inequality yields that
\begin{align*}
\int_{\s_n}\left(A_{\mu,s}f(\xi)\right)^td\sigma(\xi)
&\lesssim\int_{\s_n}\left(\int_{\widetilde{\Gamma}(\xi)}|f|^s\widehat{\mu}_{r}dv_{\alpha}\right)^{t/s}d\sigma(\xi)\\
&\leq\int_{\s_n}\left(\int_{\widetilde{\Gamma}(\xi)}|f|^qdv_{\alpha}\right)^{t/q}
    \left(\int_{\widetilde{\Gamma}(\xi)}\widehat{\mu}_r^{\frac{q}{q-s}}dv_{\alpha}\right)^{\frac{t(q-s)}{qs}}d\sigma(\xi)\\
&\leq\left(\int_{\s_n}\left(\int_{\widetilde{\Gamma}(\xi)}|f|^qdv_{\alpha}\right)^{p/q}d\sigma(\xi)\right)^{t/p}\\
&\qquad\qquad    \cdot\left(\int_{\s_n}\left(\int_{\widetilde{\Gamma}(\xi)}
    \widehat{\mu}_r^{\frac{q}{q-s}}dv_{\alpha}\right)^{\frac{pt(q-s)}{qs(p-t)}}d\sigma(\xi)\right)^{\frac{p-t}{p}}\\
&\asymp\|U_{\mu}\|^{t/s}_{L^{\frac{pt}{s(p-t)}}(\s_n)}\cdot\|f\|^t_{\HT_{q,\alpha}^p}.
\end{align*}
Therefore, $A_{\mu,s}:\HT_{q,\alpha}^p\to L^t(\s_n)$ is bounded, and
$$\|A_{\mu,s}\|_{\HT_{q,\alpha}^p\to L^t(\s_n)}\lesssim\|U_{\mu}\|^{1/s}_{L^{\frac{pt}{s(p-t)}}(\s_n)},$$
which finishes the proof.
\end{proof}

Finally, we deal with the last case of the boundedness of $A_{\mu,s}:\HT_{q,\alpha}^p\to L^t(\s_n)$.

\begin{theorem}\label{case4}
Let $\alpha>-n-1$, $p>t$ and $q\leq s$. Then $A_{\mu,s}:\HT_{q,\alpha}^p\to L^t(\s_n)$ is bounded if and only if
$$V_{\mu}(z):=\sup_{z\in\Gamma(\xi)}\widehat{\mu}_r(z)(1-|z|^2)^{\frac{(q-s)(n+1+\alpha)}{q}}$$
belongs to $L^{\frac{pt}{s(p-t)}}(\s_n)$. Moreover,
$$\|A_{\mu,s}\|_{\HT_{q,\alpha}^p\to L^t(\s_n)}\asymp\|V_{\mu}\|^{1/s}_{L^{\frac{pt}{s(p-t)}}(\s_n)}.$$
\end{theorem}
\begin{proof}
We first discuss the sufficiency. For any $f\in\HT^{p}_{q,\alpha}$, the estimate \eqref{ele2} together with \eqref{Hol} gives that
\begin{align*}
\left(A_{\mu,s}f(\xi)\right)^s&\lesssim\int_{\widetilde{\Gamma}(\xi)}|f(z)|^s\widehat{\mu}_{r}(z)dv_{\alpha}(z)\\
&\leq\left(\sup_{z\in\widetilde{\Gamma}(\xi)}\widehat{\mu}_{r}(z)(1-|z|^2)^{\frac{(q-s)(n+1+\alpha)}{q}}\right)
    \cdot\int_{\widetilde{\Gamma}(\xi)}|f(z)|^sdv_{\alpha+\frac{(s-q)(n+1+\alpha)}{q}}(z)\\
&\lesssim\left(\sup_{z\in\widetilde{\Gamma}(\xi)}\widehat{\mu}_{r}(z)(1-|z|^2)^{\frac{(q-s)(n+1+\alpha)}{q}}\right)
    \cdot\left(\int_{\widetilde{\Gamma}(\xi)}|f|^qdv_{\alpha}\right)^{s/q}.
\end{align*}
Consequently, noting that $p/t>1$, we obtain that
\begin{align*}
&\int_{\s_n}\left(A_{\mu,s}f(\xi)\right)^td\sigma(\xi)\\
&\ \lesssim\int_{\s_n}\left(\sup_{z\in\widetilde{\Gamma}(\xi)}\widehat{\mu}_{r}(z)(1-|z|^2)^{\frac{(q-s)(n+1+\alpha)}{q}}\right)^{t/s}
    \cdot\left(\int_{\widetilde{\Gamma}(\xi)}|f|^qdv_{\alpha}\right)^{t/q}d\sigma(\xi)\\
&\ \leq\left(\int_{\s_n}\left(\sup_{z\in\widetilde{\Gamma}(\xi)}\widehat{\mu}_{r}(z)
    (1-|z|^2)^{\frac{(q-s)(n+1+\alpha)}{q}}\right)^{\frac{pt}{s(p-t)}}d\sigma(\xi)\right)^{\frac{p-t}{p}}\\
&\qquad\qquad    \cdot\left(\int_{\s_n}\left(\int_{\widetilde{\Gamma}(\xi)}|f|^qdv_{\alpha}\right)^{p/q}d\sigma(\xi)\right)^{t/p}\\
&\ \asymp\|V_{\mu}\|^{t/s}_{L^{\frac{pt}{s(p-t)}}(\s_n)}\cdot\|f\|^{t}_{\HT^{p}_{q,\alpha}},
\end{align*}
which implies the boundedness of $A_{\mu,s}:\HT_{q,\alpha}^p\to L^t(\s_n)$ with 
$$\|A_{\mu,s}\|_{\HT_{q,\alpha}^p\to L^t(\s_n)}\lesssim\|V_{\mu}\|^{1/s}_{L^{\frac{pt}{s(p-t)}}(\s_n)}.$$

We next consider the necessity. By Lemma \ref{dis}, it is enough to show $\eta^{(2r)}\in T^{\frac{pt}{s(p-t)}}_{\infty}(Z)$, and in this case, we will have
$$\|V_{\mu}\|_{L^{\frac{pt}{s(p-t)}}(\s_n)}\lesssim\|\eta^{(2r)}\|_{T^{\frac{pt}{s(p-t)}}_{\infty}(Z)}.$$
Choose $\sigma>1$ such that $t\sigma>s$. We will apply the dual and factorization of sequence tent spaces as before to show $\left(\eta^{(2r)}\right)^{1/\sigma}\in T^{\frac{pt\sigma}{s(p-t)}}_{\infty}(Z)$. Let $\rho>0$ satisfy
$$\frac{1}{\rho}=\frac{1}{\sigma'}+\frac{s}{q\sigma}.$$
Then $\rho\leq1$ since $q\leq s$. Consequently, by Theorems \ref{dual} and \ref{factor},
$$T^{\frac{pt\sigma}{s(p-t)}}_{\infty}(Z)=\left(T^{\left(\frac{pt\sigma}{s(p-t)}\right)'}_{\rho}(Z)\right)^*
=\left(T^{\frac{t\sigma}{t\sigma-s}}_{\sigma'}(Z)\cdot T^{p\sigma/s}_{q\sigma/s}(Z)\right)^*.$$
Take any $c=\{c_k\}\in T^{\left(\frac{pt\sigma}{s(p-t)}\right)'}_{\rho}(Z)$ and factor it as $c_k=d_k\lambda^{s/\sigma}_k$, with $d=\{d_k\}\in T^{\frac{t\sigma}{t\sigma-s}}_{\sigma'}(Z)$ and $\lambda=\{\lambda_k\}\in T_q^p(Z)$. Then, as in the proof of Theorem \ref{case3}, we have
\begin{align*}
\left|\left\langle c,\left(\eta^{(2r)}\right)^{1/\sigma}\right\rangle_{T^2_2(Z)}\right|
&\lesssim\|d\|_{T^{\frac{t\sigma}{t\sigma-s}}_{\sigma'}(Z)}
    \cdot\|A_{\mu,s}\|^{s/\sigma}_{\HT_{q,\alpha}^p\to L^t(\s_n)}\cdot\|\lambda\|^{s/\sigma}_{T_q^p(Z)}\\
&\lesssim\|A_{\mu,s}\|^{s/\sigma}_{\HT_{q,\alpha}^p\to L^t(\s_n)}
	\cdot\|c\|_{T^{\left(\frac{pt\sigma}{s(p-t)}\right)'}_{\rho}(Z)}.
\end{align*}
By duality we obtain that $\left(\eta^{(2r)}\right)^{1/\sigma}\in T^{\frac{pt\sigma}{s(p-t)}}_{\infty}(Z)$ with
$$\left\|\left(\eta^{(2r)}\right)^{1/\sigma}\right\|_{T^{\frac{pt\sigma}{s(p-t)}}_{\infty}(Z)}\lesssim
\|A_{\mu,s}\|^{s/\sigma}_{\HT_{q,\alpha}^p\to L^t(\s_n)}.$$
Therefore,
$$\|V_{\mu}\|_{L^{\frac{pt}{s(p-t)}}(\s_n)}\lesssim\|\eta^{(2r)}\|_{T^{\frac{pt}{s(p-t)}}_{\infty}(Z)}
=\left\|\left(\eta^{(2r)}\right)^{1/\sigma}\right\|^{\sigma}_{T^{\frac{pt\sigma}{s(p-t)}}_{\infty}(Z)}\lesssim
\|A_{\mu,s}\|^{s}_{\HT_{q,\alpha}^p\to L^t(\s_n)}$$
and the proof is complete.
\end{proof}

By putting the equation \eqref{norm=}, Theorems \ref{case1}, \ref{case2}, \ref{case3} and \ref{case4} together, we obtain the whole proof of Theorem \ref{main}.

We end this section by the proof of Corollary \ref{incl}.

\begin{proof}[Proof of Corollary \ref{incl}]
By the closed graph theorem, the inclusion $\HT^p_{q,\alpha}\subset\HT^t_{s,\beta}$ is bounded whenever it is valid. Consequently, the inclusion $\HT^p_{q,\alpha}\subset\HT^t_{s,\beta}$ holds if and only if $I_d:\HT^p_{q,\alpha}\to T^t_{s}(v_{\beta+n})$ is bounded. Note that for fixed $r\in(0,1)$,
$$\widehat{(v_{\beta+n})_r}(z)\asymp(1-|z|^2)^{\beta-\alpha},\quad z\in\B_n.$$
The desired result then follows from Theorem \ref{main} easily. For example, in the case of $p>t$ and $q>s$, the inclusion $\HT^p_{q,\alpha}\subset\HT^t_{s,\beta}$ holds if and only if
$$\int_{\s_n}\left(\int_{\Gamma(\xi)}(1-|z|^2)^{\frac{q(\beta-\alpha)}{q-s}+\alpha}dv(z)\right)^{\frac{pt(q-s)}{qs(p-t)}}d\sigma(\xi)
<\infty,$$
which is equivalent to
$$\frac{q(\beta-\alpha)}{q-s}+\alpha>-n-1,$$
that is, $\frac{n+1+\alpha}{q}<\frac{n+1+\beta}{s}$.
\end{proof}

\section{Compact embeddings}\label{compact}

In this section, we characterize the compactness of $I_d:\HT^p_{q,\alpha}\to T^t_s(\mu)$. Note that by Theorem \ref{main}, if $I_d:\HT_{q,\alpha}^p\to T^t_s(\mu)$ is bounded, then $\mu$ is finite on compact subsets of $\B_n$. Hence we can assume that $\mu$ is finite on compact subsets of $\B_n$ when we deal with compactness. The main result is as follows.

\begin{theorem}\label{cpt}
Let $\alpha>-n-1$, $0<p,q,s,t<\infty$, and let $\mu$ be a positive Borel measure on $\B_n$, finite on compact subsets of $\B_n$. Fix $r\in(0,1)$. Then the following statements hold.
\begin{enumerate}
\item [(1)] If $p<t$, or $p=t$ and $q\leq s$, then $I_d:\HT_{q,\alpha}^p\to T^t_s(\mu)$ is compact if and only if
	$$\lim_{|z|\to1^-}\widehat{\mu}_r(z)(1-|z|^2)^{\frac{(q-s)(n+1+\alpha)}{q}+ns\left(\frac{1}{t}-\frac{1}{p}\right)}=0.$$
\item [(2)] If $p=t$ and $q>s$, then $I_d:\HT_{q,\alpha}^p\to T^t_s(\mu)$ is compact if and only if
	$$\widehat{\mu}_r(z)^{\frac{q}{q-s}}dv_{\alpha+n}(z)$$
	is a vanishing Carleson measure.
\item [(3)] If $p>t$ and $q>s$, then $I_d:\HT_{q,\alpha}^p\to T^t_s(\mu)$ is compact if and only if
	$$\xi\mapsto\left(\int_{\Gamma(\xi)}\widehat{\mu}_r(z)^{\frac{q}{q-s}}dv_{\alpha}(z)\right)^{\frac{q-s}{q}}$$
	belongs to $L^{\frac{pt}{s(p-t)}}(\s_n)$.
\item [(4)] If $p>t$ and $q\leq s$, then $I_d:\HT_{q,\alpha}^p\to T^t_s(\mu)$ is compact if and only if
	$$\lim_{\varrho\to1^-}\int_{\s_n}\left(\sup_{z\in\Gamma(\xi)\setminus \varrho\B_n}
	\widehat{\mu}_r(z)(1-|z|^2)^{\frac{(q-s)(n+1+\alpha)}{q}}\right)^{\frac{pt}{s(p-t)}}d\sigma(\xi)=0.$$
\end{enumerate}
\end{theorem}

Before proceeding, we establish the following little-oh version of Proposition \ref{nece}. Recall that for the $\frac{r}{2}$-lattice $Z=\{a_k\}$ and $\delta>0$, the sequence $\eta^{(\delta)}=\{\eta^{(\delta)}_k\}$ is defined by
$$\eta^{(\delta)}_k=\widehat{\mu}_{\delta}(a_k)(1-|a_k|^2)^{\frac{(q-s)(n+1+\alpha)}{q}}.$$

\begin{proposition}\label{cpt-nece}
Let $\alpha>-n-1$ and $0<p,q,s,t<\infty$. Suppose that $I_d:\HT_{q,\alpha}^p\to T^t_s(\mu)$ is compact. Then for any $\epsilon>0$, there exists $\varrho_0\in(0,1)$ such that for any $\lambda=\{\lambda_k\}\in T^p_q(Z)$,
$$\int_{\s_n}\left(\sum_{a_k\in\Gamma(\xi)\setminus\varrho_0\B_n}|\lambda_k|^s\eta^{(2r)}_k\right)^{t/s}d\sigma(\xi)\lesssim
\epsilon^t\|\lambda\|^t_{T_q^p(Z)}.$$
\end{proposition}
\begin{proof}
Fix $\epsilon>0$ and $\theta>n\max\{1,q/p,1/p,1/q\}$, and write $B_{T_q^p(Z)}$ for the closed unit ball of $T^p_q(Z)$. Since $I_d:\HT_{q,\alpha}^p\to T^t_s(\mu)$ is compact, in view of Lemma \ref{test}, we can find $\lambda^{(1)},\cdots,\lambda^{(m)}\in B_{T^p_q(Z)}$ such that for any $\lambda\in B_{T^p_q(Z)}$, there exists some $j\in\{1,\cdots,m\}$ with
$$\left\|S^{\theta}_Z(\lambda)-S^{\theta}_Z(\lambda^{(j)})\right\|_{T^t_s(\mu)}<\epsilon.$$
On the other hand, by the Lebesgue dominated convergence theorem, there is $\varrho'_0\in(0,1)$ such that for every $j\in\{1,\cdots,m\}$,
$$\int_{\s_n}\left(\int_{\Gamma(\xi)\setminus\varrho'_0\B_n}\left|S^{\theta}_Z(\lambda^{(j)})(z)\right|^s
\frac{d\mu(z)}{(1-|z|^2)^n}\right)^{t/s}d\sigma(\xi)<\epsilon^t.$$
Consequently, for any $\lambda\in B_{T^p_q(Z)}$,
\begin{align*}
&\int_{\s_n}\left(\int_{\Gamma(\xi)\setminus\varrho'_0\B_n}\left|S^{\theta}_Z(\lambda)(z)\right|^s
	\frac{d\mu(z)}{(1-|z|^2)^n}\right)^{t/s}d\sigma(\xi)\\
&\ \lesssim\left\|S^{\theta}_Z(\lambda)-S^{\theta}_Z(\lambda^{(j)})\right\|^t_{T^t_s(\mu)}\\
&\qquad  +\int_{\s_n}\left(\int_{\Gamma(\xi)\setminus\varrho'_0\B_n}\left|S^{\theta}_Z(\lambda^{(j)})(z)\right|^s
	\frac{d\mu(z)}{(1-|z|^2)^n}\right)^{t/s}d\sigma(\xi)\\
&\ \lesssim\epsilon^t,
\end{align*}
which implies that for any $\lambda\in T_q^p(Z)$,
$$\int_{\s_n}\left(\int_{\Gamma(\xi)\setminus\varrho'_0\B_n}\left|S^{\theta}_Z(\lambda)(z)\right|^s
\frac{d\mu(z)}{(1-|z|^2)^n}\right)^{t/s}d\sigma(\xi)\lesssim\epsilon^t\|\lambda\|^t_{T_q^p(Z)}.$$
Now, by the same process as in the proof of Proposition \ref{nece}, we can establish that
$$\int_{\s_n}\left(\sum_{a_k\in\Gamma(\xi)\setminus\varrho_0\B_n}|\lambda_k|^s\eta^{(2r)}_k\right)^{t/s}d\sigma(\xi)\lesssim
\epsilon^t\|\lambda\|^t_{T_q^p(Z)},$$
where $\varrho_0=\inf\{|w|:D(w,2r)\subset\B_n\setminus\varrho'_0\B_n\}$.
\end{proof}

Using Montel's theorem, it is easy to obtain that, $I_d:\HT_{q,\alpha}^p\to T^t_s(\mu)$ is compact if and only if for any bounded sequence $\{f_j\}\subset\HT^p_{q,\alpha}$ that converges to $0$ uniformly on compact subsets of $\B_n$, we have $\lim_{j\to\infty}\|f_j\|_{T^t_s(\mu)}=0$.

We are now ready to prove Theorem \ref{cpt}.

\begin{proof}[Proof of Theorem \ref{cpt}]
Fix $\epsilon>0$, and let $\{f_j\}$ be a bounded sequence in $\HT^p_{q,\alpha}$ that converges to $0$ uniformly on compact subsets of $\B_n$.

(1) This can be proved by a standard modification of the proof of Theorem \ref{case1}. The details are omitted.

(2) Suppose first that $I_d:\HT_{q,\alpha}^p\to T^t_s(\mu)$ is compact. Then by Proposition \ref{cpt-nece}, there exists $\varrho_0\in(0,1)$ such that for any $\lambda=\{\lambda_k\}\in T^p_q(Z)$,
$$\int_{\s_n}\left(\sum_{a_k\in\Gamma(\xi)}|\lambda_k|^s\eta^{(2r)}_k\chi_{(\varrho_0\B_n)^c}(a_k)\right)^{t/s}d\sigma(\xi)\lesssim
\epsilon^t\|\lambda\|^t_{T_q^p(Z)}.$$
Now, using the dual and factorization of sequence tent spaces just as in the proof of Theorem \ref{case2}, we can obtain that
$$\left\|\left\{\eta^{(2r)}_k\chi_{(\varrho_0\B_n)^c}(a_k)\right\}\right\|_{T^{\infty}_{\frac{q}{q-s}}(Z)}\lesssim\epsilon^s.$$
Let $\varrho_1=\sup\{|z|:z\in D(a_k,r),|a_k|<\varrho_0\}$. Then we have
\begin{equation}\label{rho'}
(\varrho_1\B_n)^c\subset\bigcup_{a_k\in(\varrho_0\B_n)^c}D(a_k,r).
\end{equation}
Consequently, by Lemma \ref{disCM} and its proof,
$$\left\|\chi_{(\varrho_1\B_n)^c}\nu_{\mu}\right\|_{CM}
\lesssim\left\|\left\{\eta^{(2r)}_k\chi_{(\varrho_0\B_n)^c}(a_k)\right\}\right\|^{\frac{q}{q-s}}_{T^{\infty}_{\frac{q}{q-s}}(Z)}
<\epsilon^{\frac{qs}{q-s}}.$$
Therefore,
$$\lim_{\varrho\to1^-}\left\|\chi_{(\varrho\B_n)^c}\nu_{\mu}\right\|_{CM}=0,$$
which, combined with Lemma \ref{rVCM}, implies that $\nu_{\mu}$ is a vanishing Carleson measure.

Conversely, suppose that $\nu_{\mu}$ is a vanishing Carleson measure. Then we can apply Lemma \ref{rVCM} to find $\varrho_2\in(0,1)$ such that
$$\left\|\chi_{(\varrho_2\B_n)^c}\nu_{\mu}\right\|_{CM}<\epsilon^{\frac{qs}{q-s}}.$$
This together with Lemma \ref{disCM} implies that
$$\left\|\left\{\eta^{(r/2)}_k\chi_{(\varrho'_2\B_n)^c}(a_k)\right\}\right\|_{T^{\infty}_{\frac{q}{q-s}}(Z)}
\lesssim\left\|\chi_{(\varrho_2\B_n)^c}\nu_{\mu}\right\|^{\frac{q-s}{q}}_{CM}<\epsilon^{s},$$
where $\varrho'_2=\inf\{|a_k|:D(a_k,r/2)\subset(\varrho_2\B_n)^c\}$. Since $f_j\to0$ uniformly on compact subsets of $\B_n$, we may find $J_2\geq1$ such that for $j\geq J_2$,
$$|f_j(z)|<\epsilon,\quad \forall z\in\bigcup_{a_k\in\varrho'_2\B_n}D(a_k,r).$$
Consequently, the estimate \eqref{ele1} yields that for any $\xi\in\s_n$ and $j\geq J_2$,
\begin{align*}
\left(A_{\mu,s}f_j(\xi)\right)^s
&\lesssim\left(\sum_{a_k\in\widetilde{\Gamma}(\xi)\cap\varrho'_2\B_n}+\sum_{a_k\in\widetilde{\Gamma}(\xi)\setminus\varrho'_2\B_n}\right)
    \left(\int_{D(a_k,r)}|f_j|^qdv_{\alpha}\right)^{s/q}\cdot\eta^{(r/2)}_k\\
&<\epsilon^s\sum_{a_k\in\widetilde{\Gamma}(\xi)}\left(\int_{D(a_k,r)}dv_{\alpha}\right)^{s/q}\cdot\eta^{(r/2)}_k\\
&\quad    +\sum_{a_k\in\widetilde{\Gamma}(\xi)}
    \left(\int_{D(a_k,r)}|f_j|^qdv_{\alpha}\right)^{s/q}\cdot\eta^{(r/2)}_k\chi_{(\varrho'_2\B_n)^c}(a_k),
\end{align*}
which implies that
$$\|f_j\|^t_{T^t_s(\mu)}\lesssim\epsilon^t\left\|\eta^{(r/2)}\right\|^{t/s}_{T^{\infty}_{\frac{q}{q-s}}(Z)}
+\|f_j\|^t_{\HT^p_{q,\alpha}}\cdot
\left\|\left\{\eta^{(r/2)}_k\chi_{(\varrho'_2\B_n)^c}(a_k)\right\}\right\|^{t/s}_{T^{\infty}_{\frac{q}{q-s}}(Z)}\lesssim\epsilon^t;$$
see the proof of Theorem \ref{case2}. Therefore, we obtain the desired compactness.

(3) For $\varrho\in[0,1)$, let
$$U_{\mu,\varrho}(\xi)=\left(\int_{\Gamma(\xi)\setminus\varrho\B_n}\widehat{\mu}_r(z)^{\frac{q}{q-s}}dv_{\alpha}(z)\right)^{\frac{q-s}{q}},
\quad \xi\in\s_n.$$
In view of Theorem \ref{main}, it suffices to show that $U_{\mu,0}\in L^{\frac{pt}{s(p-t)}}(\s_n)$ implies the compactness of $I_d:\HT_{q,\alpha}^p\to T^t_s(\mu)$. Since $U_{\mu,0}\in L^{\frac{pt}{s(p-t)}}(\s_n)$, the Lebesgue dominated convergence theorem yields that $\lim_{\varrho\to1^-}\left\|U_{\mu,\varrho}\right\|_{L^{\frac{pt}{s(p-t)}}(\s_n)}=0$. Therefore, by a standard modification of the proof of Theorem \ref{case3}, we can obtain that $\lim_{j\to\infty}\|f_j\|_{T^t_s(\mu)}=0$, which implies the desired compactness.

(4) For $\varrho\in[0,1)$, let
$$V_{\mu,\varrho}(\xi)=\sup_{z\in\Gamma(\xi)\setminus \varrho\B_n}\widehat{\mu}_r(z)(1-|z|^2)^{\frac{(q-s)(n+1+\alpha)}{q}}.$$
Assume first that $\lim_{\varrho\to1^-}\|V_{\mu,\varrho}\|_{L^{\frac{pt}{s(p-t)}}(\s_n)}=0$. Since $\mu$ is finite on compact subsets of $\B_n$, we obtain that $V_{\mu,0}\in L^{\frac{pt}{s(p-t)}}(\s_n)$. Then, by a standard modification of the proof of Theorem \ref{case4}, we can establish that $I_d:\HT_{q,\alpha}^p\to T^t_s(\mu)$ is compact.

We now look for the necessity. Suppose that $I_d:\HT_{q,\alpha}^p\to T^t_s(\mu)$ is compact. Then by Proposition \ref{cpt-nece}, there exists $\varrho_0\in(0,1)$ such that for any $\lambda=\{\lambda_k\}\in T^p_q(Z)$,
$$\int_{\s_n}\left(\sum_{a_k\in\Gamma(\xi)}|\lambda_k|^s\eta^{(2r)}_k\chi_{(\varrho_0\B_n)^c}(a_k)\right)^{t/s}d\sigma(\xi)
\lesssim\epsilon^t\|\lambda\|^t_{T_q^p(Z)}.$$
Using the same method as in the proof of Theorem \ref{case4}, we obtain that
$$\left\|\left\{\eta^{(2r)}_k\chi_{(\varrho_0\B_n)^c}(a_k)\right\}\right\|_{T^{\frac{pt}{s(p-t)}}_{\infty}(Z)}\lesssim\epsilon^s.$$
Therefore, it follows from \eqref{rho'} and Lemma \ref{dis} that
\begin{align*}
\|V_{\mu,\varrho_1}\|^{\frac{pt}{s(p-t)}}_{L^{\frac{pt}{s(p-t)}}(\s_n)}
&=\int_{\s_n}\left(\sup_{z\in\Gamma(\xi)\setminus\varrho_1\B_n}\widehat{\mu}_r(z)(1-|z|^2)^{\frac{(q-s)(n+1+\alpha)}{q}}\right)
    ^{\frac{pt}{s(p-t)}}d\sigma(\xi)\\
&\lesssim\int_{\s_n}\left(\sup_{a_k\in\Gamma(\xi)\setminus\varrho_0\B_n}
    \widehat{\mu}_{2r}(a_k)(1-|a_k|^2)^{\frac{(q-s)(n+1+\alpha)}{q}}\right)^{\frac{pt}{s(p-t)}}d\sigma(\xi)\\
&=\left\|\left\{\eta^{(2r)}_k\chi_{(\varrho_0\B_n)^c}(a_k)\right\}\right\|^{\frac{pt}{s(p-t)}}_{T^{\frac{pt}{s(p-t)}}_{\infty}(Z)}\\
&\lesssim\epsilon^{\frac{pt}{p-t}}.
\end{align*}
Consequently, $\lim_{\varrho\to1^-}\|V_{\mu,\varrho}\|_{L^{\frac{pt}{s(p-t)}}(\s_n)}=0$ and the proof is complete.
\end{proof}

Similar to Corollary \ref{incl}, we can use Theorem \ref{cpt} to determine when the inclusion $\HT^p_{q,\alpha}\subset\HT^t_{s,\beta}$ is compact.

\begin{corollary}
Let $\alpha,\beta>-n-1$ and $0<p,q,s,t<\infty$. Then the inclusion $\HT^p_{q,\alpha}\subset\HT^t_{s,\beta}$ is compact if and only if one of the following conditions holds:
\begin{enumerate}
	\item [(i)] $p\geq t$ and $\frac{n+1+\alpha}{q}<\frac{n+1+\beta}{s}$;
	\item [(ii)] $p<t$ and $\frac{n+1+\alpha}{q}+\frac{n}{p}<\frac{n+1+\beta}{s}+\frac{n}{t}$.
\end{enumerate}
\end{corollary}

\section{Proof of Theorem \ref{super}}\label{su}

The aim of this section is to prove Theorem \ref{super}. To this end, we need some preliminary results. The first one establishes some test functions in $\HT^p_{q,\alpha}$.

\begin{lemma}\label{test-}
Let $0<p,q<\infty$, $\alpha>-n-1$, and assume
$$\frac{n+1+\alpha}{q}<\theta<\frac{n+1+\alpha}{q}+\frac{n}{p}.$$
Then for any $\zeta\in\s_n$, the function
$$g_{\zeta}(z):=\frac{1}{(1-\langle z,\zeta\rangle)^{\theta}},\quad z\in\B_n$$
belongs to $\HT^p_{q,\alpha}$.
\end{lemma}
\begin{proof}
For any $\xi\in\s_n\setminus\{\zeta\}$, bearing in mind that $1-|z|^2\asymp|1-\langle z,\xi\rangle|$ if $z\in\Gamma(\xi)$, we may choose $\eta>\max\{n,nq/p\}$ and apply Lemma \ref{FRgeneral} to obtain that
\begin{align*}
\int_{\Gamma(\xi)}|g_{\zeta}(z)|^qdv_{\alpha}(z)
\lesssim\int_{\B_n}\frac{(1-|z|^2)^{\alpha+\eta}dv(z)}{|1-\langle z,\zeta\rangle|^{q\theta}|1-\langle z,\xi\rangle|^{\eta}}
\lesssim\frac{1}{|1-\langle \xi,\zeta\rangle|^{q\theta-n-1-\alpha}},
\end{align*}
which, in conjunction with the fact that $p\theta-p(n+1+\alpha)/q<n$, implies that
$$\int_{\s_n}\left(\int_{\Gamma(\xi)}|g_{\zeta}(z)|^qdv_{\alpha}(z)\right)^{p/q}d\sigma(\xi)
\lesssim\int_{\s_n}\frac{d\sigma(\xi)}{|1-\langle \xi,\zeta\rangle|^{p\theta-\frac{p(n+1+\alpha)}{q}}}<\infty.$$
Therefore, $g_{\zeta}\in\HT^{p}_{q,\alpha}$.
\end{proof}

The following lemma gives a necessary condition for the superposition operators between Hardy type tent spaces.

\begin{lemma}\label{poly}
Let $0<p,q,s,t<\infty$ and $\alpha,\beta>-n-1$. Suppose that $\varphi$ is a function on $\C$ such that the superposition operator $S_{\varphi}$ maps $\HT^p_{q,\alpha}$ into $\HT^t_{s,\beta}$. Then $\varphi$ is a polynomial.
\end{lemma}
\begin{proof}
For any $\delta>0$, let
$$f_{\delta}(z)=\delta z_1,\quad z=(z_1,\cdots,z_n)\in\B_n.$$
It is clear that $f_{\delta}\in\HT^{p}_{q,\alpha}$. Since $S_{\varphi}$ maps $\HT^p_{q,\alpha}$ into $\HT^t_{s,\beta}$, we know that the function $\varphi\circ f_{\delta}=S_{\varphi}f_{\delta}$ belongs to $\HT^{t}_{s,\beta}$. Consequently, the function $z\mapsto\varphi(\delta z_1)$ is holomorphic on $\B_n$, that is, the function $z_1\mapsto\varphi(\delta z_1)$ is analytic on the unit disk $\mathbb{D}$. Since $\delta>0$ is arbitrary, we obtain that $\varphi$ is an entire function on $\C$.
	
We next show that $\varphi$ is a polynomial of degree at most
$$N:=\left[\frac{pq(t(n+1+\beta)+ns)}{st(p(n+1+\alpha)+nq)}\right].$$
To this end, it is sufficient to establish that
\begin{equation}\label{Npoly}
\lim_{\varrho\to\infty}\frac{M_{\infty}(\varphi,\varrho)}{\varrho^{N+1}}=0,
\end{equation}
where $M_{\infty}(\varphi,\varrho):=\sup_{u\in\C:|u|=\varrho}|\varphi(u)|$. We complete the proof by contradiction. Assume that \eqref{Npoly} does not hold. Then there exist $\kappa>0$ and a sequence $\{u_j\}\subset\C$ with $|u_j|>1$ such that $|u_j|\to\infty$ and
\begin{equation}\label{contra}
|\varphi(u_j)|\geq \kappa|u_j|^{N+1},\quad \forall j\geq1.
\end{equation}
It is clear that for any unimodular constant $a\in\C$, the function $\varphi_{a}(u):=\varphi(au)$ also induces a superposition operator that maps $\HT^p_{q,\alpha}$ into $\HT^t_{s,\beta}$. Hence, passing to a subsequence if necessary, we may assume that $\arg u_j\to0$. Fix now $\zeta\in\s_n$ and
$$\max\left\{\frac{1}{N+1}\left(\frac{n+1+\beta}{s}+\frac{n}{t}\right),\frac{n+1+\alpha}{q}\right\}
<\theta<\frac{n+1+\alpha}{q}+\frac{n}{p},$$
and define
$$g_{\zeta}(z)=\frac{1}{(1-\langle z,\zeta\rangle)^{\theta}},\quad z\in\B_n.$$
Then by Lemma \ref{test-}, $g_{\zeta}\in\HT^{p}_{q,\alpha}$. Consequently, $\varphi\circ g_{\zeta}=S_{\varphi}g_{\zeta}\in\HT^{t}_{s,\beta}$. Since $\arg u_j\to0$, passing to a further subsequence if necessary, we may assume that $|\arg u_j|<\frac{\theta\pi}{4}$ for any $j\geq1$. For each $j\geq1$, write $v_j=1-u_j^{-1/\theta}$ and $w_j=v_j\zeta$. Since $|1-v_j|<1$ and $|\arg(1-v_j)|<\frac{\pi}{4}$, all the points $v_j$ belong to a Stolz domain with vertex at $1$. Consequently, $w_j\in\B_n$, $g_{\zeta}(w_j)=u_j$, and
\begin{equation}\label{stolz}
1-|w_j|=1-|v_j|\asymp|1-v_j|.
\end{equation}
Since $\varphi\circ g_{\zeta}\in\HT^{t}_{s,\beta}$, \cite[Lemma 2.2]{WZ22} together with \eqref{contra} and \eqref{stolz} yields that
\begin{align*}
\kappa|u_j|^{N+1}&\leq|\varphi(u_j)|=|\varphi(g_{\zeta}(w_j))|\\
&\lesssim(1-|w_j|^2)^{-\frac{n+1+\beta}{s}-\frac{n}{t}}\\
&\asymp|1-v_j|^{-\frac{n+1+\beta}{s}-\frac{n}{t}}\\
&=|u_j|^{\left(\frac{n+1+\beta}{s}+\frac{n}{t}\right)/\theta},
\end{align*}
which contradicts the fact that $|u_j|\to\infty$ since $\left(\frac{n+1+\beta}{s}+\frac{n}{t}\right)/\theta<N+1$. Therefore, $\varphi$ is a polynomial of degree at most $N$, and the proof is complete.
\end{proof}

We next characterize the monomials $\varphi$ that induce superposition operators $S_{\varphi}$ mapping $\HT^{p}_{q,\alpha}$ into $\HT^{t}_{s,\beta}$.

\begin{proposition}\label{moni}
Let $0<p,q,s,t<\infty$ and $\alpha,\beta>-n-1$. Suppose that $N$ is a non-negative integer and $\varphi(u)=u^N,u\in\C$. Then $S_{\varphi}$ maps $\HT^{p}_{q,\alpha}$ into $\HT^{t}_{s,\beta}$ if and only if $N$ satisfies
\begin{enumerate}
\item [(i)] $N\leq\frac{pq(t(n+1+\beta)+ns)}{st(p(n+1+\alpha)+nq)}$ in the case of $p(n+1+\alpha)/q<t(n+1+\beta)/s$;
\item [(ii)] $N\leq\frac{q(n+1+\beta)}{s(n+1+\alpha)}$ in the case of $p(n+1+\alpha)/q\geq t(n+1+\beta)/s$ and $\alpha\leq\beta$;
\item [(iii)] $N<\frac{q(n+1+\beta)}{s(n+1+\alpha)}$ in the case of $p(n+1+\alpha)/q\geq t(n+1+\beta)/s$ and $\alpha>\beta$.
\end{enumerate}
Moreover, if $S_{\varphi}$ maps $\HT^{p}_{q,\alpha}$ into $\HT^{t}_{s,\beta}$, then it is bounded.
\end{proposition}
\begin{proof}
If $N=0$, then it is clear that $S_{\varphi}$ maps $\HT^{p}_{q,\alpha}$ boundedly into $\HT^{t}_{s,\beta}$. So we assume that $N\geq1$. Then for any $f\in\h(\B_n)$, $f^N\in\HT^{t}_{s,\beta}$ if and only if $f\in\HT^{Nt}_{Ns,\beta}$. Moreover,
$$\|f^N\|_{\HT^{t}_{s,\beta}}=\|f\|^N_{\HT^{Nt}_{Ns,\beta}}.$$
Therefore, $S_{\varphi}$ maps $\HT^{p}_{q,\alpha}$ into $\HT^{t}_{s,\beta}$ if and only if the inclusion $\HT^{p}_{q,\alpha}\subset\HT^{Nt}_{Ns,\beta}$ holds. Consequently, we may apply Corollary \ref{incl} to obtain that $S_{\varphi}$ maps $\HT^{p}_{q,\alpha}$ into $\HT^{t}_{s,\beta}$ if and only if one of the following conditions holds:
\begin{enumerate}
	\item [($\dag$):] $p\geq Nt$, $q>Ns$ and $\frac{n+1+\alpha}{q}<\frac{n+1+\beta}{Ns}$;
	\item [($\ddag$):] $p\geq Nt$, $q\leq Ns$ and $\frac{n+1+\alpha}{q}\leq\frac{n+1+\beta}{Ns}$;
	\item [($\S$):] $p<Nt$ and $\frac{n+1+\alpha}{q}+\frac{n}{p}\leq\frac{n+1+\beta}{Ns}+\frac{n}{Nt}$.
\end{enumerate}
Note that the condition ($\S$) is trivial if $p(n+1+\alpha)/q\geq t(n+1+\beta)/s$, and the condition ($\ddag$) is trivial if $\alpha>\beta$. Hence in the case of $p(n+1+\alpha)/q\geq t(n+1+\beta)/s$ and $\alpha>\beta$, $S_{\varphi}$ maps $\HT^{p}_{q,\alpha}$ into $\HT^{t}_{s,\beta}$ if and only if ($\dag$) holds, which is in turn equivalent to $N<\frac{q(n+1+\beta)}{s(n+1+\alpha)}$. In the case of $p(n+1+\alpha)/q\geq t(n+1+\beta)/s$ and $\alpha\leq\beta$, the condition ($\dag$) is equivalent to $N<\frac{q}{s}$, and the condition ($\ddag$) is equivalent to $\frac{q}{s}\leq N\leq\frac{q(n+1+\beta)}{s(n+1+\alpha)}$, so $S_{\varphi}$ maps $\HT^{p}_{q,\alpha}$ into $\HT^{t}_{s,\beta}$ if and only if $N\leq\frac{q(n+1+\beta)}{s(n+1+\alpha)}$. The desired result for the case $p(n+1+\alpha)/q<t(n+1+\beta)/s$ can be obtained similarly. Since Corollary \ref{incl} ensures that the inclusion $\HT^{p}_{q,\alpha}\subset\HT^{Nt}_{Ns,\beta}$ is bounded whenever it is valid, we establish that $S_{\varphi}:\HT^{p}_{q,\alpha}\to\HT^{t}_{s,\beta}$ is bounded whenever $S_{\varphi}$ maps $\HT^{p}_{q,\alpha}$ into $\HT^{t}_{s,\beta}$.
\end{proof}

The following lemma establishes a comparison principle for superposition operators, which indicates that if $\varphi$ is a polynomial such that $S_{\varphi}$ maps $\HT^{p}_{q,\alpha}$ into $\HT^{t}_{s,\beta}$, then for any polynomial $\psi$ whose degree is at most the degree of $\varphi$, $S_{\psi}$ also maps $\HT^{p}_{q,\alpha}$ into $\HT^{t}_{s,\beta}$.

\begin{lemma}\label{comp}
Let $0<p,q,s,t<\infty$ and $\alpha,\beta>-n-1$, and let $\varphi$ be a polynomial of degree $k$. Suppose that $S_{\varphi}$ maps $\HT^{p}_{q,\alpha}$ into $\HT^{t}_{s,\beta}$. Then for any polynomial $\psi$ of degree $j$ with $0\leq j\leq k$, $S_{\psi}$ maps $\HT^{p}_{q,\alpha}$ into $\HT^{t}_{s,\beta}$. Moreover, there exists $C=C(\varphi,\psi,s,t,\beta)>0$ such that for any $f\in\HT^{p}_{q,\alpha}$,
$$\|S_{\psi}f\|_{\HT^{t}_{s,\beta}}\leq C\left(\|S_{\varphi}f\|_{\HT^{t}_{s,\beta}}+1\right).$$
\end{lemma}
\begin{proof}
Since $j\leq k$, there exist $C,R>0$ such that for any $u\in\C$ with $|u|\geq R$, $|\psi(u)|\leq C|\varphi(u)|$. Fix $f\in\HT^{p}_{q,\alpha}$. For any $\xi\in\s_n$, write
$$A_{\xi}:=\{z\in\Gamma(\xi):|f(z)|\geq R\}.$$
Noting that $\beta>-n-1$ ensures that $v_{\beta}(\Gamma(\xi))\asymp v_{n+\beta}(\B_n)<\infty$, we obtain that
\begin{align*}
\int_{\Gamma(\xi)}|S_{\psi}f(z)|^sdv_{\beta}(z)
&=\left(\int_{A_{\xi}}+\int_{\Gamma(\xi)\setminus A_{\xi}}\right)|\psi(f(z))|^sdv_{\beta}(z)\\
&\lesssim C^s\int_{\Gamma(\xi)}|\varphi(f(z))|^sdv_{\beta}(z)+M^s_{\infty}(\psi,R),
\end{align*}
which implies that
\begin{align*}
\int_{\s_n}\left(\int_{\Gamma(\xi)}|S_{\psi}f(z)|^sdv_{\beta}(z)\right)^{t/s}d\sigma(\xi)
&\lesssim\int_{\s_n}\left(\int_{\Gamma(\xi)}|\varphi(f(z))|^sdv_{\beta}(z)\right)^{t/s}d\sigma(\xi)+1\\
&=\|S_{\varphi}f\|^t_{\HT^{t}_{s,\beta}}+1,
\end{align*}
where the implicit constant depends only on $\varphi,\psi,s,t$ and $\beta$. Therefore, $S_{\psi}f\in\HT^{t}_{s,\beta}$, and
$$\|S_{\psi}f\|_{\HT^{t}_{s,\beta}}\lesssim\|S_{\varphi}f\|_{\HT^{t}_{s,\beta}}+1.$$
The arbitrariness of $f$ finishes the proof.
\end{proof}

We are now in a position to prove Theorem \ref{super}.

\begin{proof}[Proof of Theorem \ref{super}]
If $\varphi$ is a polynomial of degree $N$ that satisfies the desired conditions, then in view of Proposition \ref{moni}, $S_{\varphi}$ maps $\HT^{p}_{q,\alpha}$ boundedly into $\HT^{t}_{s,\beta}$. Since for any $z\in\B_n$, the point evaluation at $z$ is bounded on $\HT^{t}_{s,\beta}$, we can apply \cite[Theorem 3.1]{BR} to obtain that $S_{\varphi}:\HT^{p}_{q,\alpha}\to\HT^{t}_{s,\beta}$ is continuous. Conversely, suppose that $S_{\varphi}$ maps $\HT^{p}_{q,\alpha}$ into $\HT^{t}_{s,\beta}$. Then by Lemma \ref{poly}, $\varphi$ is a polynomial. Assume the degree of $\varphi$ is $N$. Then Lemma \ref{comp} gives that the function $\psi(u)=u^N,\ u\in\C$ also induces a superposition operator that maps $\HT^{p}_{q,\alpha}$ into $\HT^{t}_{s,\beta}$. Consequently, by Proposition \ref{moni}, $N$ satisfies the desired conditions.
\end{proof}







\begin{thebibliography}{99}

\bibitem{AMV} V. \'{A}lvarez, M. A. M\'{a}rquez and D. Vukoti\'{c},
\newblock{Superposition operators between the Bloch space and Bergman spaces,}
\newblock Ark. Mat. 42 (2004), no. 2, 205--216.

\bibitem{AZ} J. Appell and P. P. Zabrejko,
\newblock{Nonlinear superposition operators,}
\newblock Cambridge Tracts in Mathematics, 95. Cambridge University Press, Cambridge, 1990. 

\bibitem{Ars} M. Arsenovi\'{c},
\newblock{Embedding derivatives of $\mathcal{M}$-harmonic functions into $L^p$ spaces,}
\newblock Rocky Mountain J. Math. 29 (1999), no. 1, 61--76.

\bibitem{BR} C. Boyd and P. Rueda,
\newblock{Holomorphic superposition operators between Banach function spaces,}
\newblock J. Aust. Math. Soc. 96 (2014), no. 2, 186--197.

\bibitem{Ca95} G. A. C\'{a}mera,
\newblock{Nonlinear superposition on spaces of analytic functions,} 
\newblock Harmonic analysis and operator theory (Caracas, 1994), 103--116, Contemp. Math., 189, Amer. Math. Soc., Providence, RI, 1995.

\bibitem{CaG94} G. A. C\'{a}mera and J. Gim\'{e}nez,
\newblock{The nonlinear superposition operator acting on Bergman spaces,}
\newblock Compositio Math. 93 (1994), no. 1, 23--35.

\bibitem{Car0} L. Carleson,
\newblock{An interpolation problem for bounded analytic functions,}
\newblock Amer. J. Math. 80 (1958), 921--930.

\bibitem{Car1} L. Carleson,
\newblock{Interpolations by bounded analytic functions and the corona problem,}
\newblock Ann. of Math. 76 (1962), 547--559.

\bibitem{Ch23} J. Chen,
\newblock{Closures of holomorphic tent spaces in weighted Bloch spaces,}
\newblock Complex Anal. Oper. Theory 17 (2023), no. 6, Paper No. 87, 20 pp.

\bibitem{CPW} J. Chen, J. Pau and M. Wang,
\newblock{Essential norms and Schatten(--Herz) classes of integration operators from Bergman spaces to Hardy spaces,}
\newblock Results Math. 76 (2021), no. 2, Paper No. 88, 33 pp.

\bibitem{Co} W. S. Cohn,
\newblock{Generalized area operators on Hardy spaces,}
\newblock J. Math. Anal. Appl. 216 (1997), no. 1, 112--121.

\bibitem{CV} W. S. Cohn and I. E. Verbitsky,
\newblock{Factorization of tent spaces and Hankel operators,}
\newblock J. Funct. Anal. 175 (2000), no. 2, 308--329.

\bibitem{CMS} R. R. Coifman, Y. Meyer and E. M. Stein,
\newblock{Some new function spaces and their applications to harmonic analysis,}
\newblock J. Funct. Anal. 62 (1985), no. 2, 304--335.

\bibitem{DoGi} S. Dom\'{i}nguez and D. Girela,
\newblock{Superposition operators between mixed norm spaces of analytic functions,}
\newblock Mediterr. J. Math. 18 (2021), no. 1, Paper No. 18, 11 pp.

\bibitem{Du69} P. L. Duren,
\newblock{Extension of a theorem of Carleson,}
\newblock Bull. Amer. Math. Soc. 75 (1969), 143--146.

\bibitem{Du} P. L. Duren,
\newblock{Theory of $H^p$ spaces,}
\newblock Academic Press, New York-London, 1970.

\bibitem{GLW} M. Gong, Z. Lou and Z. Wu,
\newblock{Area operators from $\mathcal{H}^p$ spaces to $L^q$ spaces,}
\newblock Sci. China Math. 53 (2010), no. 2, 357--366.

\bibitem{H} L. H\"{o}rmander,
\newblock{$L^p$ estimates for (pluri-)subharmonic functions,}
\newblock Math. Scand. 20 (1967), 65--78.

\bibitem{Jev} M. Jevti\'{c},
\newblock{Embedding derivatives of $\mathcal{M}$-harmonic Hardy spaces $\mathcal{H}^p$ into Lebesgue spaces, $0<p<2$,}
\newblock Rocky Mountain J. Math. 26 (1996), no. 1, 175--187.

\bibitem{KS88} R. Kerman and E. Sawyer,
\newblock{Carleson measures and multipliers of Dirichlet-type spaces,}
\newblock Trans. Amer. Math. Soc. 309 (1988), no. 1, 87--98.

\bibitem{LLZ} X. Liu, Z. Lou and R. Zhao,
\newblock{Area operators on Hardy spaces in the unit ball of $\mathbb{C}^n$,}
\newblock J. Math. Anal. Appl. 513 (2022), no. 2, Paper No. 126222, 22 pp.

\bibitem{Lue91} D. Luecking,
\newblock{Embedding derivatives of Hardy spaces into Lebesgue spaces,}
\newblock Proc. London Math. Soc. (3) 63 (1991), no. 3, 595--619.

\bibitem{Lue93} D. Luecking,
\newblock{Embedding theorems for spaces of analytic functions via Khinchine's inequality,}
\newblock Michigan Math. J. 40 (1993), no. 2, 333--358.

\bibitem{LvP2} X. Lv and J. Pau,
\newblock{Tent Carleson measures for Hardy spaces,}
\newblock J. Funct. Anal. 287 (2024), no. 2, Paper No. 110459, 27 pp.

\bibitem{LvPW} X. Lv, J. Pau and M. Wang,
\newblock{Area operators on Bergman spaces,}
\newblock Acta Math. Sin. (Engl. Ser.) 40 (2024), no. 5, 1161--1176.

\bibitem{Me22} T. Mengestie,
\newblock{Weighted superposition operators on Fock spaces,}
\newblock Rev. R. Acad. Cienc. Exactas F\'{i}s. Nat. Ser. A Mat. RACSAM 116 (2022), no. 2, Paper No. 75, 10 pp.

\bibitem{MPPW} S. Miihkinen, J. Pau, A. Per\"{a}l\"{a} and M. Wang,
\newblock{Volterra type integration operators from Bergman spaces to Hardy spaces,}
\newblock J. Funct. Anal. 279 (2020), no. 4, Paper No. 108564, 32 pp.

\bibitem{OF} J. M. Ortega and J. F\'{a}brega,
\newblock{Pointwise multipliers and corona type decomposition in $BMOA$,}
\newblock Ann. Inst. Fourier (Grenoble) 46 (1996), no. 1, 111--137.

\bibitem{Pa20} C. Parks,
\newblock{Interpolation and sampling in analytic tent spaces,}
\newblock preprint, (2020). \url{https://doi.org/10.48550/arXiv.2001.06152}

\bibitem{P1} J. Pau,
\newblock{Integration operators between Hardy spaces on the unit ball of $\mathbb{C}^n$,}
\newblock J. Funct. Anal. 270 (2016), no. 1, 134--176.

\bibitem{PP10} J. Pau and J. \'{A}. Pel\'{a}ez,
\newblock{Embedding theorems and integration operators on Bergman spaces with rapidly decreasing weights,}
\newblock J. Funct. Anal. 259 (2010), no. 10, 2727--2756.

\bibitem{P2018} J. Pau and A. Per\"{a}l\"{a},
\newblock{A Toeplitz-type operator on Hardy spaces in the unit ball,}
\newblock Trans. Amer. Math. Soc. 373 (2020), no. 5, 3031--3062.

\bibitem{PZ15} J. Pau and R. Zhao,
\newblock{Carleson measures and Toeplitz operators for weighted Bergman spaces on the unit ball,}
\newblock Michigan Math. J. 64 (2015), no. 4, 759--796.

\bibitem{PR14} J. \'{A}. Pel\'{a}ez and J. R\"{a}tty\"{a},
\newblock{Weighted Bergman spaces induced by rapidly increasing weights,}
\newblock Mem. Amer. Math. Soc. 227 (2014), no. 1066, vi+124 pp.

\bibitem{PR} J. \'{A}. Pel\'{a}ez and J. R\"{a}tty\"{a},
\newblock{Embedding theorems for Bergman spaces via harmonic analysis,}
\newblock Math. Ann. 362 (2015), no. 1-2, 205--239.

\bibitem{Per} A. Per\"{a}l\"{a},
\newblock{Duality of holomorphic Hardy type tent spaces,}
\newblock preprint, (2018). \url{https://doi.org/10.48550/arXiv.1803.10584}

\bibitem{SW} F. Sun and H. Wulan,
\newblock{The nonlinear superposition operators between some analytic function spaces,}
\newblock Proc. Amer. Math. Soc. 151 (2023), no. 8, 3431--3437.

\bibitem{WZ21} M. Wang and L. Zhou,
\newblock{Carleson measures and Toeplitz type operators on Hardy type tent spaces,}
\newblock Complex Anal. Oper. Theory 15 (2021), no. 4, Paper No. 70, 46 pp.

\bibitem{WZ22} M. Wang and L. Zhou,
\newblock{Embedding derivatives and integration operators on Hardy type tent spaces,}
\newblock Acta Math. Sin. (Engl. Ser.) 38 (2022), no. 6, 1069--1093.

\bibitem{Wu06} Z. Wu,
\newblock{Area operator on Bergman spaces,}
\newblock Sci. China Ser. A 49 (2006), no. 7, 987--1008.

\bibitem{Wu11} Z. Wu,
\newblock{Volterra operator, area integral and Carleson measure,}
\newblock Sci. China Math. 54 (2011), no. 11, 2487--2500.

\bibitem{ZZ} R. Zhao and K. Zhu,
\newblock{Theory of Bergman spaces in the unit ball of $\mathbb{C}^n$,}
\newblock M\'{e}m. Soc. Math. Fr. (N.S.) No. 115 (2008), vi+103 pp.

\bibitem{Zhuball} K. Zhu,
\newblock{Spaces of holomorphic functions in the unit ball,}
\newblock Springer-Verlag, New York, 2005.

\end{thebibliography}
\end{document}